\documentclass{amsart}

\usepackage[dvips]{graphicx}
\usepackage{amssymb}
\usepackage[all,2cell,dvips]{xy} \UseAllTwocells \SilentMatrices
\usepackage[usenames]{color}

\newcommand\C{\mathbb{C}}
\newcommand\R{\mathbb{R}}
\newcommand\HH{\mathbb{H}}
\newcommand\Q{\mathbb{Q}}
\newcommand\Qbar{\overline{\mathbb{Q}}}
\newcommand\A{\mathbb{A}}
\newcommand\p{\mathfrak{p}}
\newcommand\m{\mathfrak{m}}
\renewcommand\O{\mathcal{O}}
\renewcommand\P{\mathbb{P}}
\newcommand\Sp{\mathbb{S}}
\newcommand\Z{\mathbb{Z}}
\newcommand\F{\mathbb{F}}
\newcommand\PSL{\mathop{\rm PSL}\nolimits}
\newcommand\Isom{\mathop{\rm Isom}\nolimits}
\newcommand\Tr{\tilde{X}}
\newcommand\PTr{\tilde{Y}}

\newcommand\Tra{\tilde{X}_{\rm a}}
\newcommand\Trna{\tilde{X}_{\rm na}}

\newcommand\Homna{\mathop{\rm R}\nolimits_{\rm na}}
\newcommand\Homnahat{\mathop{\hat{\rm R}}\nolimits_{\rm na}}
\newcommand\Hom{{\rm R}}
\newcommand\Homhat{\hat{\rm R}}

\newcommand\SL{\mathop{\rm SL}\nolimits}
\newcommand\tr{\mathop{\rm tr}\nolimits}

\newcommand\isom{\cong}
\newcommand\Ge{\Gamma_{\rm e}}
\newcommand\Tp{T_{\rm e}}
\newcommand\lam{\lambda}

\newtheorem{theorem}{Theorem}[section]
\newtheorem{lemma}[theorem]{Lemma}
\newtheorem{proposition}[theorem]{Proposition}
\newtheorem{conjecture}[theorem]{Conjecture}
\newtheorem{question}[theorem]{Question}
\newtheorem{corollary}[theorem]{Corollary}
\newtheorem{definition}[theorem]{Definition}
\newtheorem{remark}[theorem]{Remark}

\newcounter{nootje}
\title{On Character varieties of two-bridge knot groups}
\author{Melissa L. Macasieb \\ 
Kathleen L. Petersen\\
Ronald M. van Luijk}
\address{Department of Mathematics,
Mathematics Building,
University of Maryland,
College Park, MD 20742}
\email{melmacasieb@gmail.com}
\address{Department of Mathematics,
         Florida State University, 208 Love Building, 
         Tallahassee, FL 32306-4510}
\email{petersen@math.fsu.edu}
\address{Mathematisch Instituut, Universiteit Leiden, Postbus 9512,
2300 RA, Leiden, The Netherlands}
\email{rmluijk@gmail.com}

\begin{document}

\begin{abstract} We find explicit models for the
  $\PSL_2(\C)$- and $\SL_2(\C)$-character varieties of the fundamental
  groups of complements in $\Sp^3$ of an infinite
  family of two-bridge knots that contains the twist knots. 
  We compute the genus of the components of these
  character varieties, and deduce upper bounds on the degree of the
  associated trace fields.  We also show that these knot complements
  are fibered if and only if they are commensurable to a fibered knot
  complement in a $\Z/2\Z$-homology sphere, resolving a conjecture of Hoste and
  Shanahan. 
\end{abstract}

\maketitle

\section{Introduction}
Given a finitely generated group $\Gamma$, the set of all representations
$\Gamma \to \SL_2(\C)$ naturally carries the structure of an algebraic set.
So does the set of characters of these representations. Often the
components of this last set that contain only characters of abelian 
representations are 
well understood. The union of the other components is called 
the $\SL_2(\C)$-character variety of $\Gamma$.
Over the last few decades, the $\SL_2(\C)$-character
variety of the fundamental groups of hyperbolic $3$-manifolds has 
proven to be an effective tool in understanding their
topology (see \cite{CCGLS}, \cite{CGLS}, \cite{CS}). The same
can be said for their $\PSL_2(\C)$-character variety, 
defined in \S \ref{psltwo}, but 
in general it is difficult to find even the simplest 
invariants of these varieties, such as the number of irreducible components.

\vspace{.1cm}
\begin{figure}[h!]
\centering
\includegraphics[scale=.6]{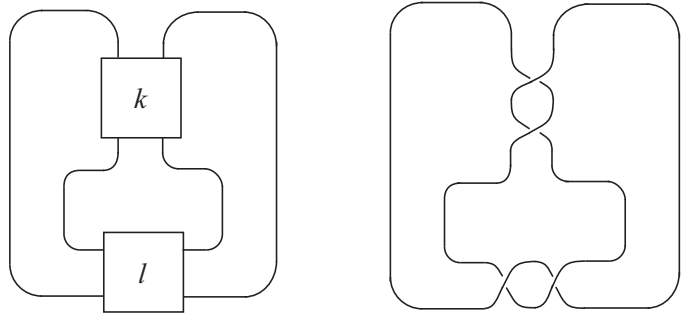}
\caption{The knot $J(k,l)$ and the figure-eight knot $J(2,-2)$.}
\label{twisty}
\end{figure}

In this paper we consider the case that $\Gamma$ is a knot group, i.e., 
the fundamental group of the complement in $\Sp^3$ of a knot. 
We look at the knots $J(k,l)$ 
as described in Figure \ref{twisty}, where $k$ and $l$ are integers 
denoting the number of half twists in the labeled boxes; positive
numbers correspond to right-handed twists and negative numbers
correspond to left-handed twists. Note that $J(k,l)$ is a knot if and 
only if $kl$ is even; otherwise it is a two-component link. 
The subfamilies of knots $J(\pm 2,l)$, with $l \in \Z$, consist of all 
twist knots, containing the figure-eight knot $J(2,-2)$ and the trefoil
$J(2,2)$. The complement of the knot $J(k,l)$ is hyperbolic if and only if 
$|k|,|l| \geq 2$ and $J(k,l)$ is not the trefoil.


We compute the genus of every component of the
character varieties associated to these knots. 
This is the first 
time such results have been found for an infinite family of knots.
In particular it shows that the genus of both character varieties of 
a knot complement can be arbitrarily large, which was not known before. 

More precisely, for any nonzero integers $k$ and $l$ with $kl$ even, 
we let $M(k,l)$ denote 
the complement $\Sp^3 \setminus J(k,l)$ and let $X(k,l)$ and $Y(k,l)$ 
denote the $\SL_2(\C)$- and $\PSL_2(\C)$-character variety of 
the fundamental group $\pi_1(M(k,l))$. Both varieties are curves 
and $X(k,l)$ is a double cover of $Y(k,l)$. 
Our first main result is a non-recursively defined model for $Y(k,l)$.
Secondly, we construct a projective birational model for $Y(k,l)$ that 
we prove to be smooth and irreducible when $J(k,l)$ is hyperbolic and $k\neq l$. For $k=l>2$ the
curve $Y(k,l)$ has two smooth components and we identify which of the 
two is the 
canonical component $Y_0(k,l)$, defined in \S \ref{psltwo}. 
The results, and those for $X(k,l)$ and its canonical component $X_0(k,l)$, 
defined in \S \ref{sltwo}, are summarized in the following theorems. 

\begin{theorem}\label{mainone}
  Let $k,l$ be any nonzero integers with $l$ even, $|k|\geq 2$, and
  $k\neq l$.
\begin{enumerate}
\item The curve $Y(k,l)$ is irreducible. It has geometric genus
  \[ (\lfloor |k|/2\rfloor-1)(|l|/2-1)\] and is hyperelliptic if and only
  if $|k|\leq 5$ or $|l|\leq 5$.
\item If $|l|>2$, then the curve $Y(l,l)$ has two components. The component
$Y_0(l,l)$ has genus $0$. The other component has genus $(|l|/2-2)^2$ and is 
hyperelliptic if and only if $|l|\leq 6$.
\end{enumerate}
\end{theorem}

\begin{theorem}\label{maintwo}
  Suppose $l$ is a nonzero even integer, say $l=2n$.  If $k\neq l$ is
  an integer satisfying $|k|\geq 2$, then $X(k,l)$ is irreducible and
  its genus equals
$$
3|mn|-|m|-a|n|+b,
$$
with $m = \lfloor k/2 \rfloor$ and
$$
a = \left\{
\begin{array}{rl}
  4 & \mbox{if $k$ is odd and $k<0$},\cr
  1 & \mbox{otherwise}.\cr
\end{array}
\right.
\qquad
b = \left\{
\begin{array}{rl}
  2 & \mbox{if $k$ is odd and $k<0<l$},\cr
  1 & \mbox{if $k$ is odd and $l<0$},\cr
  -1 & \mbox{if $k$ is even and $kl>0$},\cr
  0 & \mbox{otherwise}.\cr
\end{array}
\right.
$$
If $|l|>2$, then $X(l,l)$ has two components, namely $X_0(l,l)$ of
genus $|n|-1$ and an other component of genus $3n^2-7|n|+5$.
\end{theorem}
Precisely two knots in this family have canonical components of their 
$\SL_2(\C)$-character varieties that have genus $1$, namely 
the figure-eight knot $J(2,-2)$ and the $7_4$ knot $J(4,4)$.

Recent results have shown that arithmetic properties of the
$\SL_2(\C)$- and $\PSL_2(\C)$-character varieties can give 
information about topological invariants
such as the commensurability classes of knot complements (\cite{CD},
\cite{HS}, \cite{LR}).  
Our irreducibility results allow us to use  a
criterion of Calegari and Dunfield \cite{CD} to prove a
conjecture of Hoste and Shanahan \cite[Conj. 1]{HS} about 
commensurability classes of the knots $J(k,l)$. Note that fibered means 
fibered over $\Sp^1$. The result is the following.

\begin{theorem}\label{mainfour}
The manifold $M(k,l)$ is fibered if and only if  $M(k,l)$ is
commensurable to a fibered knot complement in a $\Z / 2\Z$-homology
sphere. 
\end{theorem}

%
If $K$ is a hyperbolic knot, 
let  $[F(K) : \Q]$ denote the degree of the trace field
$F(K)$ of $K$ over $\Q$, i.e., the field generated 
by all traces of elements in the image of a lift 
$\pi_1(\Sp^3\setminus K)\to \SL_2(\C)$ of the discrete faithful 
representation (see \cite{HStrace} and \S \ref{sltwo}). 
%
From the non-recursively defined model for $Y(k,l)$ 
we can deduce an upper bound for the degree of the trace field of $J(k,l)$.
The following theorem  says that for all hyperbolic $J(k,l)$ this 
bound is of the same order of magnitude as the genus of $X_0(k,l)$. 

\begin{theorem}\label{mainthree}
Let $k$ and $l$ be integers for which $J(k,l)$ is a hyperbolic knot.
Then the degree $[F\big(J(k,l)\big) : \Q]$ of the trace 
field of $J(k,l)$ is bounded by $\frac{1}{2}|kl|$.
It is bounded by 
$\frac{1}{2}kl-1$ if $kl>0$ and by $|l|-1$ if $k=l$.
\end{theorem}

In \S \ref{twobridgeknots} we define the family of two-bridge knots
$K(p,q)$, parametrized by pairs $(p,q)$ of coprime odd integers
satisfying $-p < q \leq p$. For all nonzero integers $k,l$ with $kl$
even, the knot $J(k,l)$ is ambient isotopic with $K(p,q)$ for the
unique such $p,q$ for which the image of $q/p$ in $\Q/\Z$ equals
that of  $l/(1-kl)$; we find from the roughest bound in 
Theorem \ref{mainthree} that 
$(p-1)/2$ is an upper bound for the degree of the trace field of 
$K(p,q)$. This also follows for 
general two-bridge knots from a result of Riley \cite[\S 3]{rileytwo}.

For any nonzero integers $k$ and $l$ with $kl$ even, let 
 $c(k,l)$ denote the crossing number of the knot $J(k,l)$, i.e., the 
minimum number of crossings in any projection of the knot. 
For the hyperbolic twist knots $J(2,l)$ the smallest bounds of 
Theorem \ref{mainthree} are in fact equalities 
and directly related to the crossing number $c(2,l)$ by
\cite[Thm. 1, Cor. 1]{HStrace}.
We immediately obtain an interesting corollary. 

\begin{corollary}\label{twisttrace}
For any integer $l \neq -1,0,1,2$ the genus of the 
$\SL_2(\C)$-character variety $X(2,l)=X_0(2,l)$ of $J(2,l)$ equals 
$$
c(2,l )-3 = [F \big(J(2,l)\big) : \Q] - 1.
$$ 
\end{corollary}

It is easy to check the degree of the trace field of $J(k,l)$ 
for small values of $|k|$ and $|l|$, where the smallest upper 
bounds given in Theorem \ref{mainthree}
are in fact equalities. We therefore wonder the following. 

\begin{question}
Let $k$ and $l$ be integers for which $J(k,l)$ is a hyperbolic knot. Is
the degree $[F\big(J(k,l)\big) : \Q]$ of the trace 
field of $J(k,l)$ equal to $-\frac{1}{2}kl$ if $kl<0$? Is it equal to
$\frac{1}{2}kl-1$ if $kl>0$ and 
$k \neq l$ and equal to $|l|-1$ if $k=l$?
\end{question}

In fact, for all $p,q$ as above with $p < 100$ and $K(p,q)$
hyperbolic, we checked that 
when the character variety of the two-bridge knot $K(p,q)$ is
irreducible, then the degree of the trace field $F\big(K(p,q)\big)$ 
of $K(p,q)$ equals
the upper bound $(p-1)/2$ proven by Riley. We therefore also wonder
the following. 

\begin{question}
Let $p$ and $q$ be coprime odd integers with $-p<q<p$ for which the
knot $K(p,q)$ is hyperbolic. Assume that the $\PSL_2(\C)$-character variety
of the fundamental group of the complement of $K(p,q)$ is irreducible. 
Is the degree of the trace field of $K(p,q)$ equal to $(p-1)/2$?
\end{question}

The paper is set up as follows. In the next section we describe
character varieties in general and in particular 
for two-bridge knots, a family of knots that contains our family.
This includes the definition of the canonical component.
In \S \ref{famtwobridge} we describe the family $J(k,l)$ as 
a subfamily of the two-bridge knots and find the fundamental groups 
of their complements. In \S \ref{algebra} we give a brief summary
of the theory of Newton polygons and algebraic curves. 

The two models for $Y(k,l)$ are defined in \S \ref{equationsection}
and \S \ref{newmodelsection}. More precisely, the standard model $C(k,l)$ 
is given non-recursively in Proposition \ref{Ckl} and the smooth model 
$D(k,l)$ is given in (\ref{defeqsnew}). The birationality is proven in 
Proposition \ref{charvar}. Proposition \ref{Yzero} identifies which 
component of the new model of $Y_0(l,l)$ corresponds with the canonical 
component, after which we can prove Theorem \ref{mainthree}.

We find the 
number of components of $Y(k,l)$ for all integers $k$ and $l$ and prove 
that all components are smooth in \S \ref{smoothirred}. In \S \ref{generairred} we use this 
to prove Theorems \ref{mainone} and \ref{maintwo}. Theorem \ref{mainfour}
is proved in the final section, \S \ref{commens}.

\section{Preliminaries}

\subsection{Representation and character varieties}
\label{charvarprelim}

We will begin with some background material concerning the
representation and character varieties of finitely generated groups,
and knot groups in particular.  Standard references for this material
are \cite{CGLS} and \cite{CS}.

Let $\Gamma$ be any finitely generated group with generating set
$\{\gamma_1, \ldots, \gamma_N\}$. The set $\Hom(\Gamma)={\rm
  Hom}(\Gamma,\SL_2(\C))$ can be given the structure of an affine
algebraic set defined over $\Q$ by using the entries of the images of
the $\gamma_i$ under $\rho \in \Hom(\Gamma)$ as coordinates for
$\rho$. We therefore will refer to $\Hom(\Gamma)$ as the {\em
  $\SL_2(\C)$-representation variety} of $\Gamma$. The isomorphism
class of this variety does not depend on the choice of generators.  In
general, $\Hom(\Gamma)$ need not be irreducible.

\subsubsection{$\SL_2(\C)$-character varieties}\label{sltwo}

The {\rm character} of a representation $\rho$ is the function
$\chi_{\rho}\colon \Gamma \to \C$ defined by $\chi_{\rho}(\gamma) =
\tr (\rho(\gamma))$.
%
%
Define the set of characters $\Tr(\Gamma) = \{\chi_\rho : \rho \in
\Hom(\Gamma)\}$, which is often denoted by $X(\Gamma)$ elsewhere in
the literature, but we will reserve that notation for a particular
subset of $\Tr(\Gamma)$.
%
%
For all $\gamma \in \Gamma$ we define the function $t_\gamma \colon
\Hom(\Gamma) \rightarrow \C$ by $t_{\gamma}(\rho) =
\chi_{\rho}(\gamma)$.  Let $T$ be the subring of the ring of all
functions from $\Hom(\Gamma)$ to $\C$ that is generated by $1$ and the
functions $t_\gamma$ for $\gamma \in \Gamma$.  The ring $T$ is
finitely generated, for instance by the elements
\[
t_{\gamma_{i_1} \cdots \gamma_{i_r}}, \,\, 1\leq i_1<\ldots <i_r \leq N
\]
(see \cite{CS}, Proposition 1.4.1). This implies that a character
$\chi \in \Tr(\Gamma)$ is determined by its values on finitely many
elements of $\Gamma$. If $h_1, \ldots, h_m$ are generators of $T$,
then the map $\Hom(\Gamma) \rightarrow \C^m$ given by $\rho \mapsto
(h_1(\rho), \ldots, h_m(\rho))$ induces an injection $\Tr(\Gamma)
\rightarrow \C^m$.  This gives $\Tr(\Gamma)$ the structure of a closed
algebraic subset of $\C^m$, but the fact that it is closed is quite
nontrivial (see \cite[Proposition 1.4.4]{CS}). 
It follows that $\Tr(\Gamma)$ has the structure of an abstract affine
algebraic variety with coordinate ring $T_\C = T\otimes \C$. Different
sets of generators of $T$ give different models for $\Tr(\Gamma)$, all
isomorphic over $\Z$.  We refer to $\Tr(\Gamma)$ as the {\em
  $\SL_2(\C)$-character variety} of $\Gamma$.

A representation $\rho \in \Hom(\Gamma)$ is {\em reducible} if all the
$\rho(\gamma)$ with $\gamma \in \Gamma$ have a common one-dimensional
eigenspace, otherwise it is called {\em irreducible}.  A
representation $\rho$ is {\em abelian} if its image is an abelian
subgroup of $\SL_2(\C)$, and {\em nonabelian} otherwise.  Note that
every irreducible representation is necessarily nonabelian, although
there do exist nonabelian reducible representations.  For 
fundamental groups of knot complements in $\Sp^3$ these are all 
metabelian (see \cite[Section 1]{HLM}).  

The group $\SL_2(\C)$ acts on $\Hom(\Gamma)$ by conjugation.  Let
$\Homhat(\Gamma)$ denote the set of orbits.  Two representations
$\rho, \rho' \in \Hom(\Gamma)$ are {\em conjugate} if they lie in the
same orbit.  Since two conjugate representations give the same
character, the trace map $\Hom(\Gamma) \rightarrow \Tr(\Gamma)$
induces a well-defined map $\Homhat(\Gamma) \rightarrow \Tr(\Gamma)$.
Note that if $\Gamma$ is finite, then this map is a bijection, but in
general it need not be injective. It is injective when restricted to
irreducible representations; if $\rho, \rho' \in \Hom(\Gamma)$ have
equal characters $\chi_\rho = \chi_{\rho'}$, and $\rho$ is
irreducible, then $\rho$ and $\rho'$ are conjugate (see 
\cite[Proposition 1.5.2]{CS}).

Let $\Tra(\Gamma)$, and $\Trna(\Gamma)$ denote the set of characters
of abelian and nonabelian representations $\rho \in \Hom(\Gamma)$
respectively. The set $\Tra(\Gamma)$ is a Zariski closed subset of
$\Tr(\Gamma)$ (see \cite[Propositions 1.3(ii) and 1.7(1)]{HLM}). 

We can say more when $\Gamma$ is the fundamental group of a knot 
complement in $\Sp^3$.  We will {\bf assume} this to be case from now on, say
$\Gamma = \pi_1(M)$ is the fundamental group of the $3$-manifold
$M=\Sp^3\setminus K$ for the knot $K$ in $\Sp^3$. Then $\Trna(\Gamma)$ is also
a Zariski closed subset of $\Tr(\Gamma)$ and $\Tra(\Gamma)$ is
isomorphic to $\A^1$ (see \cite[Proposition 1.7(2) and Corollary
1.10]{HLM}).  
As the characters of abelian
representations are well understood, we will focus only on
$\Trna(\Gamma)$, which we will also denote by $X(\Gamma)$.  By abuse
of language, we will refer to $X(\Gamma)$ as the {\em
  $\SL_2(\C)$-character variety} of $\Gamma$ as well.


If $M$ is a hyperbolic knot complement, then $M$ is isomorphic to a 
quotient of hyperbolic $3$-space $\HH^3$ by a discrete group. 
By Mostow-Prasad rigidity there is then a discrete faithful 
representation $\overline{\rho}_0 \colon \Gamma \hookrightarrow 
\Isom^+(\HH^3) \isom \PSL_2(\C)$ that is unique up to conjugation, 
defining an action of 
$\Gamma$ on $\HH^3$ whose quotient $\HH^3/\Gamma$ is isomorphic with $M$.
Moreover, the representation $\overline{\rho}_0$ can be lifted to a 
discrete faithful representation $\Gamma \hookrightarrow 
\SL_2(\C)$. Fix such a lift and call it $\rho_0$. 
By work of Thurston \cite{T}, 
the character of $\rho_0$ is contained in a unique component 
of $X(\Gamma)$, which has dimension $1$ and which will be 
denoted by $X_0(\Gamma)$. In all cases presented in this paper, 
we will see that $X_0(\Gamma)$ does not depend on the choice 
of lift $\rho_0$. 


\subsubsection{$\PSL_2(\C)$-character varieties}\label{psltwo}

There are various constructions for the $\PSL_2(\C)$-representation
and character varieties of $\Gamma$, none of which are quite as
standard.  We refer the reader to the general definition in 
\cite[\S 2.1]{LR} and to \cite[\S 3]{BoZ}, and \cite{GAM}. Since in our case
$\Gamma$ is the fundamental group of a knot complement in 
$\Sp^3$, the definitions simplify dramatically. Note
that $\mu_2\isom \{\pm 1\}$ is isomorphic to the kernel of the
homomorphism $\SL_2(\C) \rightarrow \PSL_2(\C)$.

The first simplification comes from the fact that we have $H^2(\Gamma,
\mu_2) =0$ (see \cite[page 756]{BoZ}, or \cite[remark after Lemma 2.1]{GAM}.
Under this condition, the $\PSL_2(\C)$-{\em
  character variety} $\PTr(\Gamma)$ is isomorphic to the quotient
$\Tr(\Gamma) / {\rm Hom}(\Gamma,\mu_2)$, where $\sigma \in {\rm
  Hom}(\Gamma,\mu_2)$ acts on $\chi_\rho \in \Tr(\Gamma)$ by $(\sigma
\chi_\rho)(\gamma) = \sigma(\gamma)\chi_\rho(\gamma)$ for all $\gamma
\in \Gamma$.

The second simplification comes from a better understanding of ${\rm
  Hom}(\Gamma,\mu_2)$ in our specific case.  Since $\Gamma$ is a knot
group, there are presentations for $\Gamma$ where the generators
$\gamma_i$ are all meridians of $K$. For such a presentation, the
$\gamma_i$ are all conjugate and we have $t_{\gamma_i} = t_{\gamma_j}$
for $1\leq i,j\leq N$.  In fact, the Wirtinger presentation (see
\cite[Section 3.D]{rolfsen}) is such a presentation where the
relations are of length $4$ in the generators and their inverses, with
one relation for each crossing.  Therefore, there is a well-defined
notion of {\em parity} of an element $\gamma \in \Gamma$, based on the
parity of the length of $\gamma$ as a word in terms of meridians.  Let
$\Ge \subset \Gamma$ denote the subgroup of index $2$ consisting of
all even $\gamma \in \Gamma$. Any $\sigma \in {\rm Hom}(\Gamma,\mu_2)$
sends all the (conjugate) meridians to the same element, so $\sigma$
is trivial on $\Ge$, and we find ${\rm Hom}(\Gamma,\mu_2) \isom {\rm
  Hom}(\Gamma/\Ge,\mu_2) \isom {\rm Hom}(\mu_2,\mu_2) \isom
\mu_2$. The induced action of $\mu_2$ on $R(\Gamma)$ is given by
$(-\rho)(\gamma) = -\rho(\gamma)$ for $\gamma \not \in \Ge$ and
$(-\rho)(\gamma) = \rho(\gamma)$ for $\gamma \in \Ge$. The induced
action on $\Tr(\Gamma)$ is given by $-\chi_\rho = \chi_{-\rho}$, and
the corresponding action on $T$ by negating $t_\gamma$ for all $\gamma
\not \in \Ge$.  We conclude that the $\PSL_2(\C)$-{character variety}
$\PTr(\Gamma)$ is isomorphic to $\Tr(\Gamma)/\mu_2$ and its coordinate
ring is $\Tp \otimes \C$, where $\Tp=T^{\mu_2}$ is the subring of $T$
of all elements invariant under $\mu_2$.

We let $Y(\Gamma)$ denote the image of $X(\Gamma) = \Trna(\Gamma)$
under the quotient map $\Tr(\Gamma) \rightarrow \PTr(\Gamma)$. As for
$X(\Gamma)$, by abuse of language, we will refer to $Y(\Gamma)$ as the
$\PSL_2(\C)$-{\em character variety} of $\Gamma$. If $M$ is
hyperbolic, then we denote the component of $Y(\Gamma)$ that contains
the character of the discrete faithful representation of $\Gamma$ by
$Y_0(\Gamma)$, obtaining a map $X_0(\Gamma) \rightarrow Y_0(\Gamma)$.

\subsection{Character varieties of two-bridge knot complements}\label{repvar}


The knots $J(k,l)$ that we are interested in are part of a larger 
family, the so-called two-bridge knots. As we will use some results
on two-bridge knots, we now describe these knots and their 
character varieties.

\subsubsection{Two-bridge knots}\label{twobridgeknots}
two-bridge knots are those knots admitting a projection with only two
maxima and two minima. To every two-bridge knot we can associate a
pair $(p,q)$ of coprime odd integers with $-p < q \leq p$, such that
the two-bridge knot is ambient isotopic to the knot $K(p,q)$ we now
define.  As described in \cite[Chapter 12]{BuZ}, to a pair $(p,q)$ as
above, we associate the sequence $[a_1, \ldots, a_s]$ of entries in
the continued fraction
$$
\frac{q}{p} +\epsilon=
\cfrac{1}{a_1+\cfrac{1}{a_2+\cfrac{1}{a_3+\cfrac{1}{\cdots +
        \cfrac{1}{a_s}}}}}  
$$
where $\epsilon \in \{0,1\}$ is such that $0<\frac{q}{p} +\epsilon
\leq 1$ and where these entries satisfy $a_i \geq 1$ and they are
chosen such that $s$ is odd, which is possible by replacing the last
entry $a$ of the usual continued fraction by the two elements $a-1$
and $1$ if necessary.  Then $K(p,q)$ is the knot presented by the
so-called $4$-plat in Figure \ref{fourplatgen}, where the $j$-th block
between the two middle strands consists of $a_{2j-1}$ left-handed half
twists, and the $j$-th block between the two left-most strands
consists of $a_{2j}$ right-handed half twists.  The knots $K(p,q)$ and
$K(p',q')$ (with $(p,q)$ and $(p',q')$ as above) 
are ambient isotopic if and only if $p=p'$ and either
$q=q'$ or $qq' \equiv 1 \pmod p$ (see \cite[Theorem
12.6]{BuZ}); if $p=p'$ and $qq'
\equiv 1 \pmod p$, then the $4$-plat presentation of $K(p',q')$ is
obtained from turning the $4$-plat presentation of $K(p,q)$ upside
down, i.e., reversing the sequence $[a_1, \ldots, a_s]$, which comes 
down to rotating about a ``horizontal'' line in $\Sp^3$. 
Indeed, it is well known that the fractions $q/p$ and
$q'/p'$ of the continued fractions associated to any sequence of
numbers of odd length and its reverse respectively, satisfy $p=p'$ and
$qq' \equiv 1 \pmod p$.

\vspace{.1cm}
\begin{figure}[h!]
\centering
\includegraphics[scale=.6]{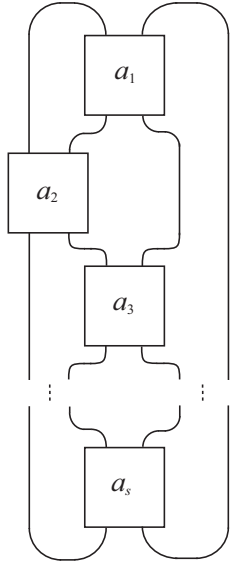}
\caption{The $4$-plat corresponding to $[a_1, a_2, \dots, a_s]$ for $s$ odd.}
\label{fourplatgen}
\end{figure}

If $p=p'$ and $qq' \equiv 1 \pmod p$, then turning the $4$-plat
$K(p,q)$ upside down induces isomorphisms between the fundamental
groups and character varieties of $K(p,q)$ and $K(p',q')$.  Now assume
$q=q'$, so $q^2 \equiv 1 \pmod p$, and let $Y(p,q)$ denote the
$\PSL_2(\C)$-character variety associated to $K(p,q)$. Then turning
the $4$-plat presentation upside down induces an automorphism of
$Y(p,q)$. If furthermore $K(p,q)$ is hyperbolic, then Ohtsuki \cite{Ohtsuki} 
proves that $Y(p,q)$ is
reducible by showing that the canonical component $Y_0(p,q)$ is fixed
by this involution, while other components are not. 
This fact will be used in \S \ref{newmodelsection} to 
determine $Y_0(2n,2n)$.

The fundamental group $\pi_1 (\Sp^3 \setminus K(p,q))$ of the knot complement 
$\Sp^3 \setminus K(p,q)$ has a presentation 
\begin{equation}\label{pres1}
\Gamma =\langle \, a, b\,\,| \,\, wa= bw \, \rangle,
\end{equation}
where 
\begin{equation}\label{wform}
w= a^{e_1} b^{e_2} \cdots a^{e_{p-2}}b^{e_{p-1}}
\end{equation}
with $e_i = (-1)^{\lfloor \frac{iq}{p}\rfloor}$.  This presentation
follows from the canonical Schubert normal form \cite{Schub} of the two-bridge
diagram of $K(p,q)$ (see \cite[Prop. 1]{rileytwo}, 
\cite[(2.1)]{murasugi}, \cite[Prop. 1]{mayland}).

\subsubsection{Character Varieties}\label{charvars}
As in the previous section, for any $\gamma \in \Gamma$, let
$t_\gamma$ be the function $t_\gamma \colon \Hom(\Gamma) \rightarrow
\C, \rho \mapsto \tr(\rho(\gamma))$, and let $T$ be the subring of the
ring of all functions from $\Hom(\Gamma)$ to $\C$ that is generated by
$1$ and these functions. Since $a$ and $b$ are conjugate in $\Gamma$,
we have $t_a =t_b$. Therefore, the ring $T$ is generated by $t_a$ and
$t_{ab}$ (see \S \ref{sltwo}), which are the most common traces
used as coordinates to define the $\SL_2(\C)$-character variety of
$K(p,q)$. We will use slightly different coordinates, which define a
nicer model.  For any $\rho \in \Hom(\Gamma)$, the matrices $\rho(b)$
and $\rho(b^{-1})$ have the same traces, so we have
$t_{b^{-1}}=t_b=t_a$. Using $a$ and $b^{-1}$ as generators of
$\Gamma$, we may also use $t_a$ and $t_{ab^{-1}}$ as coordinates.
Therefore, the $\SL_2(\C)$-character variety $X(\Gamma)$ may be
identified with the image of $\Homna(\Gamma)$ under the map
$(t_{ab^{-1}},t_a) \colon \Hom(\Gamma) \rightarrow \A^2$, where
$\Homna(\Gamma)$ is the set of nonabelian representations.  For any
$\lam_0 \in \C^*$ and $r_0\in \C$, we set
$$
A(\lam_0) =\left( \begin{array}{cc} \lam_0 & 1 \\ 0 & \lam_0^{-1}
\end{array} \right), \quad B(\lam_0,r_0)=\left( \begin{array}{cc} 
\lam_0 & 0 \\ 2-r_0 & \lam_0^{-1} \end{array} \right).
$$
The entry $2-r_0$ in $B(\lam_0,r_0)$ is chosen so that 
$A(\lam_0)B(\lam_0,r_0)^{-1}$ has trace $r_0$. 

\begin{proposition}\label{choosesmart}
  Let $\rho \in \Hom(\Gamma)$ be a nonabelian representation.  Then
  there are $\lambda_0 \in \C^*$ and $r_0\in \C$ such that $\rho$ is
  conjugate to the representation $\rho'$ determined by
  $\rho'(a)=A(\lam_0)$ and $\rho'(b)=B(\lam_0,r_0)$.  Conversely, any
  representation $\rho'$ of this form is nonabelian and
  $(t_{ab^{-1}},t_a)(\rho')=(r_0,\lam_0+\lam_0^{-1})$.
\end{proposition}
\begin{proof}
Since $a$ and $b$ are conjugate in $\Gamma$, they have the same trace. 
This and the fact that they do not commute is enough to conclude the 
first statement by \cite[Lemma 7]{rileythree}.
For the second statement, suppose that $\rho'$
satisfies the given conditions. Then we have $t_{ab^{-1}}(\rho') =
\tr(\rho'(ab^{-1})) =r_0$, and $t_a(\rho') = \tr(\rho'(a))
=\lam_0+\lam_0^{-1}$. If $\rho'$ were abelian, then from
$\rho'(w)\rho'(a) = \rho'(b)\rho'(w)$ we would find $\rho'(a) =
\rho'(b)$, which is a contradiction. This finishes the proof.
\end{proof}

A nonabelian representation $\rho$ is irreducible if and only if the
$r_0$ in Proposition \ref{choosesmart} satisfies $r_0 \neq 2$.
Consider a point $P=(r_0,x_0) \in \A^2$. By Proposition
\ref{choosesmart}, the point $P$ is contained in $X(\Gamma)$ if and
only if there is a $\lam_0 \in \C^*$ with $x_0=\lam_0+\lam_0^{-1}$
such that the assignments $a \mapsto A(\lam_0)$ and $b\mapsto
B(\lam_0,r_0)$ can be extended to a representation $\rho \in
\Hom(\Gamma)$. Choose either $\lambda_0$ for which we have $x_0 =
\lam_0+\lam_0^{-1}$, and let $W(\lam_0,r_0)$ denote the right-hand
side of (\ref{wform}) with $A(\lam_0)$ and $B(\lam_0,r_0)$ substituted
for $a$ and $b$ respectively. Then the assignment extends to a
representation if and only if we have $W(\lam_0,r_0)A(\lam_0) =
B(\lam_0,r_0)W(\lam_0,r_0)$, which results in four equations in
$\lambda_0$ and $r_0$. The following proposition states that these
equations reduce to a single equation in $r_0$ and $x_0$, which is
therefore independent of the choice of $\lambda_0$.  (Note that
$x_0=\lambda_0+\lambda_0^{-1}$ and so $x_0^2-2 =
\lam_0^2+\lambda_0^{-2}$.)

\begin{proposition}\label{eqFprop}
  Consider the ring $\Q[r,\lambda,\lambda^{-1}]$ and let $I$ denote
  the ideal generated by the four entries of the matrix
  $W(\lam,r)A(\lam) - B(\lam,r)W(\lam,r)$. Then $I$ is generated by
\begin{equation}\label{eqF}
F = W_{11}+(\lam^{-1} - \lam) W_{12}, 
\end{equation}
where $W_{ij}$ denotes the $(i,j)$-entry of $W(\lam,r)$.  Moreover, if
we set $y=\lambda^2+\lambda^{-2}$, then $F$ is contained in the
subring $\Q[r,y]$ of $\Q[r,\lambda,\lambda^{-1}]$.  
\end{proposition}
\begin{proof}
See \cite[Theorem 1]{Ri}. 
\end{proof}

We conclude that $X(\Gamma)$ is given in $\A^2(r,x)$ by $F=0$, where
$F$ is viewed as a polynomial in $x=\lambda+\lambda^{-1}$.  In
particular, if $K(p,q)$ is hyperbolic, then the canonical component
$X_0(\Gamma)$ will be an irreducible component of this algebraic set.

The coordinate ring of $X(\Gamma)$ is $\C[r,x]/(F)$ with $F$ as in
Proposition \ref{eqFprop}. Note again that $r$ and $x$ correspond to
$t_{ab^{-1}}$ and $t_a$.  The involution $\chi \rightarrow -\chi$ from
\S \ref{psltwo} fixes $r$ and sends $x$ to $-x$. This implies
that the coordinate ring of $Y(\Gamma)$ is isomorphic to the subring
$\C[r,x^2]/(F) \isom \C[r,y]/(F)$, with $y = x^2-2$ corresponding to
$t_{a^2}$.  That is, $Y(\Gamma)$ is given in $\A^2(r,y)$ by $F=0$ with
$F$ viewed as a polynomial in $y=\lambda^2+\lambda^{-2}$. Therefore
the double cover $X(\Gamma)\rightarrow Y(\Gamma)$ is given by $(r,x)
\mapsto (r,x^2-2)$.

The projective closure of this model of $Y(\Gamma)$ has bad
singularities at infinity. We will see that in the case of the
subfamily of two-bridge knots of the form $J(k,l)$, discussed in the
next section, there is an other model of $Y(\Gamma)$, whose
coordinates are $t_{ab^{-1}}$ and the trace of another element, that
has a smooth projective closure in $\P^1 \times \P^1$. This will allow
us, for instance, to compute the geometric genus of the irreducible
components of $Y(\Gamma)$ and $X(\Gamma)$ for that family.

\begin{remark}
  The trace map $\Homnahat(\Gamma) \rightarrow X(\Gamma)$ from the set
  of conjugacy classes of nonabelian representations in $\Hom(\Gamma)$
  to the set of their characters is injective when restricted to
  irreducible representations, as discussed in \S
  \ref{charvarprelim}. The reader be warned, however, that for
  reducible representations this is not the case.
%
%
  As stated correctly in \cite{Bu2}, a representation of the form
  mentioned in Proposition \ref{choosesmart} with $\lambda_0$ and
  $r_0$ is conjugate to the representation of the same form with
  $\lambda_0^{-1}$ and $r_0$, but in general only in a group larger
  than $\SL_2(\C)$.  For $r_0=0$ these representations are not
  conjugate in $\SL_2(\C)$, while they do have the same characters.
\end{remark}

\subsection{A family of two-bridge  knots}\label{famtwobridge}

We are interested in the family of knots of the form $J(k,l)$ as
described in the introduction (see Figure \ref{twisty}).  Note that
$J(k,l)$ is a knot precisely when $kl$ is even, which we will almost
always assume to be the case. Note also that $J(k,l)$ is symmetric in
$k$ and $l$. We will often make use of this symmetry and assume that
$l$ is even. Note furthermore that there is an obvious rotation of
$\Sp^3$ taking $J(k,l)$ to its reverse when $l$ is even, and that
$J(-k,-l)$ is the mirror image of $J(k,l)$. Sometimes we will use this
to assume without loss of generality that $k$ or $l$ is nonnegative.
These are not the only equivalences among the knots, as for any
integer $l$ the knots $J(2,l)$ and $J(-2,l-1)$ are equivalent.

If $kl$ is even, then the knot $J(k,l)$ is ambient isotopic to the
two-bridge knot $K(p,q)$ for the unique odd and coprime integers $p$
and $q$ with $-p<q\leq p$ for which the image of $\frac{q}{p}$ in 
$\Q/\Z$ equals that of $\frac{l}{1-kl}$. 
Note that for any integer $l$ the knots
$J(2,l)$ and $J(-2,l-1)$ give the same $p$ and $q$. 
When $|k|$ and $|l|$ are large enough, the
following table shows to which sequence of numbers the corresponding
$4$-plat is associated.
$$
\begin{array}{ll}
[1,k-2,1,l-2,1] & \text{for } k,l>2, \cr
[1,k-1,-l] & \text{for $k>1$ and $l<0$,} \cr
[-k,l-1,1] & \text{for $k<0$ and $l>1$,} \cr
[-k-1,1,-l-1] & \text{for } k,l<-1.
\end{array}
$$
Indeed, these $4$-plats are easily checked to be 
ambient isotopic with $J(k,l)$ (see Figure \ref{4plat1} for 
the case $k,l>2$). 
%
%
The remaining cases have small $|k|$ or $|l|$ and are also easily
checked. Note that $J(l,k)$ is ambient isotopic with $K(p',q')$ for
$p',q'$ coprime odd integers such that $-p' < q' \leq p'$ and
$\frac{q'}{p'} = \frac{k}{1-kl}$ in $\Q/\Z$. Then we have $p=p'$ and
$qq' \equiv 1 \pmod p$, so switching $k$ and $l$ corresponds with
turning the $4$-plat upside down, cf. \S \ref{twobridgeknots},
and $k=l$ implies $q=q'$.


\vspace{.1cm}
\begin{figure}[h!]
\centering
\includegraphics[scale=.6]{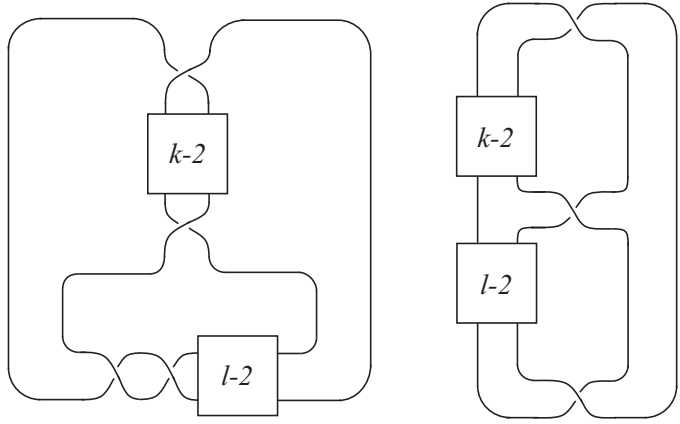}
\caption{4-plat presentation of $J(k,l)$ for $k,l\geq 2$}
\label{4plat1}
\end{figure}

For any integers $k,l$, let $\pi_1(k,l)$ denote the fundamental group
of of $\Sp^3 \setminus J(k,l)$.  By Proposition 1 of \cite{HS2}, for even $l$,
say $l=2n$, this group has a presentation 
\begin{equation}\label{pres2}
  \pi_1(k,2n) \isom \langle \, a,b \,\, | \,\, a w_k^{n} 
= w_k^{n} b \, \rangle
\end{equation}
with
\begin{equation}\label{wk}
w_k = \left\{ 
\begin{array}{ll}
(ab^{-1})^m(a^{-1}b)^m  & \mbox{if } k = 2m, \\
(ab^{-1})^mab(a^{-1}b)^m  & \mbox{if } k = 2m+1. \\
\end{array}
\right.
\end{equation}
We will sketch a proof here, as we need a little more information
about the structure of $\pi_1(k,l)$. We will also prove that for
$l=2n+1$, the group has a presentation
\begin{equation}\label{pres3}
  \pi_1(k,2n+1) \isom \langle \, a,b \,\, | \,\, a w_k^nb 
= w_k^{n+1} \, \rangle.
\end{equation}

\vspace{.1cm}
\begin{figure}[h!]
\centering
\includegraphics[scale=.75]{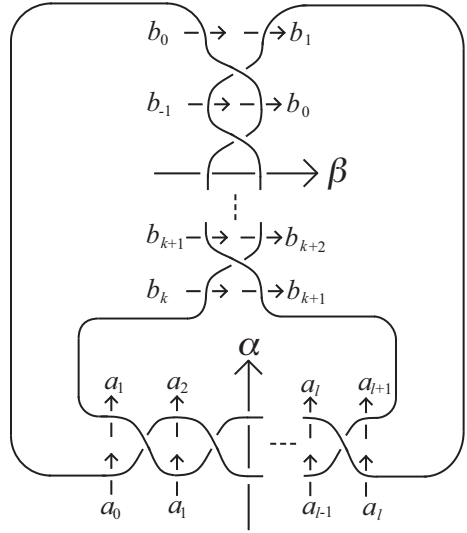}
\caption{Generators for $\pi_1(k,l)$ with $k<0<l$}
\label{newfancypicture}
\end{figure}

As in \cite[Section 3.D]{rolfsen}, where this is made precise, 
we interpret Figure
\ref{newfancypicture} as a knot, contained almost entirely
in one plane, except for the crossings, and with the base point
$P$ at "the eye of the reader."
For $0 \leq j \leq l+1$ in case $l>0$ and for $l \leq j \leq 1$ in
case $l<0$, we let $a_j$ be the loop based at $P$ that consists of 
the line segment from $P$ to the tail of the arrow labeled $a_j$, 
followed by the arrow itself and the segment from the head of the 
arrow to $P$. Similarly, for all appropriate $j$ we let
$b_j$ be the loop associated to the arrow labeled $b_j$.  
The product $xy$ of two
loops $x$ and $y$ based at $P$ is the compositum of the two loops,
where we first follow $x$ and then $y$.  
Set $a=a_0$, $b=b_1$, $\alpha=a_0a_1$, and
$\beta=b_0b_1$. Then by induction (downwards if $k$ or $l$ is negative) we
have $a_j = \alpha^{-d} a_{j-2d} \alpha^d$ for $d=\lfloor j/2 \rfloor$ and $b_j
= \beta^{-d} b_{j-2d} \beta^d$ for $d=\lfloor j/2 \rfloor$ for each
appropriate $j$. Using this and the identity $b_0=a_0^{-1}=a^{-1}$, we
can express $b_j$ in terms of $a$ and $b$ for each $j$. Using
$a_1=b_k$ we can then also express $a_j$ in terms of $a$ and $b$ for
each $j$. We find $\beta=a^{-1}b$ and $\alpha=w_k$ with $w_k$ as in (\ref{wk}),
in terms of the meridians $a$ and $b$. We are left with two relations
in terms of $a$ and $b$, namely $a_l=b_1$ and $a_{l+1} =
b_{k+1}^{-1}$, which are dependent, as we have
$$
a_la_{l+1} = \alpha = a_0a_1 = b_0^{-1}b_k =
b_1(b_0b_1)^{-1}(b_kb_{k+1})b_{k+1}^{-1} = b_1\beta^{-1}\beta b_{k+1}^{-1} =
b_1b_{k+1}^{-1}.
$$
It follows that the fundamental group is generated by the elements $a$
and $b$ with the relation $b_1=a_l$. For even $l$, say $l=2n$, this
relation is $b=\alpha^{-n}a_0\alpha^n$, or $w_k^nb = aw_k^n$. For odd $l$, say
$l=2n+1$, the relation is $b=\alpha^{-n}a_1\alpha^n = \alpha^{-n}a_0^{-1}
(a_0a_1)\alpha^n = \alpha^{-n}a^{-1}\alpha^{n+1}$, or $aw_k^nb = w_k^{n+1}$. This shows that the
fundamental group can indeed be presented as claimed.

Now let $a_i'$, $b_j'$, $\alpha'$, and $\beta'$ be the analogous loops for the
knot $J(l,k)$ and set $b'=b_1'$ and $a'=a_0'$.  Then there is a
natural isomorphism from $\pi_1(k,l)$ to $\pi_1(l,k)$ that sends $a_j$
to $b_j'$, $b_i$ to $a_i'$, and $\alpha=w_k(a,b)$ and $\beta=a^{-1}b$ to
$\beta'=a'^{-1}b'$ and $\alpha'=w_l(a',b')$ respectively.  This isomorphism is
induced by turning the $4$-plat associated to $J(k,l)$ upside down to
obtain that of $J(l,k)$.  The elements $\alpha$ and $\beta$ will play an
important role in the new model of the $\PSL_2(\C)$-character variety
of $J(k,l)$ that we will define later. 

We leave it to the reader to check that the group presentations
(\ref{pres1}) coming from the Schubert normal form and the
presentations (\ref{pres2}) and (\ref{pres3}) of $\pi_1(k,l)$ are
equivalent in case $kl$ is even.  For even $l$ an isomorphism is given
by sending $a$ to $a$ and $b$ to $b$, while for odd $l$ (and thus even
$k$) an isomorphism is given by sending $a$ to $a$ and $b$ to
$b^{-1}$. 

We set $X(k,l)=X(\pi_1(k,l))$ and define $Y(k,l)$ similarly, as well
as $X_0(k,l)$ and $Y_0(k,l)$ in case $J(k,l)$ is a hyperbolic knot.

\subsection{Newton Polygons and Algebraic curves}\label{algebra}

\subsubsection{Discrete valuations and Newton polygons}
In the proof of our main theorem we will make heavy use of valuations.
A {\em non-archimedean} valuation on a field $K$ is a map
$v \colon \, K \to \R \cup \{\infty\}$ with 
$v(x) = \infty \Leftrightarrow x=0$ that 
satisfies the ultrametric triangle  
inequality $v(x+y) \geq \min\big(v(x),v(y)\big)$ and $v(xy) =
v(x)+v(y)$ for all $x,y \in K$. Given such a valuation $v$ on $K$, the set 
$R_v = \{x\in K \,\,: \,\, v(x) \geq 0 \}$ forms a subring of $K$
that is a local ring with maximal ideal $\m_v = \{x\in K \,\,: \,\,
v(x) > 0 \}$. For any $x,y \in K$ with $v(x) < v(y)$ we have 
$v(x+y) = v(x)$.  For any real number $\alpha$ with $0<\alpha < 1$ we
obtain an absolute value $|\cdot |_v \colon K \to
\R_{\geq 0}$ by setting $|x|_v = \alpha^{v(x)}$. 
For more details, see \cite[Ch. 2]{goss} and \cite[\S I.1-2, \S
II.1-3]{localfields}. 

An example of a non-archimedean valuation is the $p$-adic valuation
$v_p$ on $\Q$; for any nonzero integer $a$, the valuation $v_p(a)$ equals 
the number of factors $p$ in $a$, and for any two nonzero integers 
$a,b$ we have $v_p(a/b) = v_p(a)-v_p(b)$. 
By definition this valuation extends uniquely to a valuation, also
denoted by $v_p$, on the completion $\Q_p$ of $\Q$ at $v_p$, the field
of $p$-adic numbers, containing the associated local ring $\Z_p$ of
$p$-adic integers. We can also extend $v_p$, though not
necessarily uniquely, to any finite extension of $\Q$ or $\Q_p$, and
by taking limits also to any algebraic extension of $\Q$ or $\Q_p$. 
Note that for any such extension $v$ of
$v_p$ we have $v(p^{1/n})=1/n$ for any nonzero integer $n$, so 
the values of a valuation are not necessarily integral.

Let $v$ be a non-archimedean valuation on a field $K$ and 
$f=\sum_{i=0}^n a_i x^i \in K[x]$ a nonzero polynomial. Then the 
{\em Newton polygon} of $f$ at $v$ is the lower convex hull of the
$n+1$ points $\big(i,v(a_i)\big)$, where the point is at infinity if
$a_i=0$. Note that if $a_0=a_1=\ldots=a_{i-1}=0$ and $a_i \neq 0$ for 
some $i>0$, then the left-most segment of the Newton polygon is the 
vertical segent from $(0,\infty)$ to $\big(i,v(a_i)\big)$, which 
has horizontal length $i$.
The following lemma tells us that the Newton polygon
determines the valuations of the roots of $f$. 

\begin{lemma}\label{polyroots}
Let $v$ be a non-archimedean valuation on an algebraically closed
field $K$ and $f \in K[x]$ a nonzero polynomial. 
Then for any rational number $q$, the number of roots of
$f$ in $K$ with valuation $q$ equals the horizontal
length of the segment of the Newton polygon of $f$ at $v$ 
with slope $-q$ if such a segment exists, and it equals $0$ otherwise.
\end{lemma}
\begin{proof}
See \cite[Prop. 2.9]{goss}.
\end{proof}

\subsubsection{Algebraic curves}

In this section we will assume that the ground field is algebraically
closed. 
For the basic properties of algebraic varieties, in particular curves,
and the notions of rational maps and morphisms between them, we refer
the reader to \cite[Ch. I-II]{silv}. The topology we use on algebraic
varieties is the Zariski topology, which on curves is the cofinite
topology. We stress the fact that a 
{\em rational map} $\varphi \colon \, C\dashrightarrow D$ of varieties
is given by rational functions on $C$ and not necessarily defined on
the whole of $C$; the map $\varphi$ is a {\em morphism} if it is regular
everywhere on $C$ and $\varphi$ is called {\em birational} if it
restricts to an isomorphism from a nonempty open subset of $C$ to an
open subset of $D$. In particular, two curves are birational if they
are isomorphic up to a finite number of points.   

\begin{lemma}\label{Csmooth}
  Suppose $\varphi \colon C \rightarrow D$ is a birational morphism of
  curves.  If $D$ is smooth, then $C$ is isomorphic to $\varphi(C)$.
\end{lemma}
\begin{proof}
Let $C'$ be a projective closure of $C$ and let 
$\psi\colon \, U \to C\subset C'$ be a birational inverse of $\varphi$
with $U\subset \varphi(C)$ open. Since $\varphi(C)$ is smooth and $C'$
projective, the map $\psi$ extends to a morphism $\hat{\psi}\colon \,
\varphi(C) \to C'$ \cite[Prop. I.6.8]{H}. The composition 
$\hat{\psi} \circ \varphi \colon C \to C'$ is the identity on a
dense open subset of $C$, so it is the identity on $C$. It follows 
that $\varphi$ induces an isomorphism from $C$ to $\varphi(C)$.
\end{proof}

We call a curve {\em hyperelliptic} if it is birational to a double
cover of $\P^1$. Note that with this definition, all curves of genus
$0$ and $1$ are hyperelliptic. 
For $i\in \{1,2\}$, let $\pi_i \colon \P^1 \times \P^1 \to \P^1$
denote the projection on the $i$-th factor. 
If $C\subset \P^1 \times \P^1$ is a curve, then for almost all $P \in
\P^1$ the number of intersection points between $\P^1 \times \{P\}$
and $C$ equals the degree $\deg \pi_2|_C$ of the map 
$\pi_2|_C \colon \, C \to \P^1$ induced by $\pi_2$; the {\em bidegree}
of $C$ is the pair of integers $(\deg \pi_2|_C, \deg \pi_1|_C)$. 
Two curves $C,C' \subset \P^1 \times \P^1$ of bidegree $(a,b)$ and
$(a',b')$ respectively have intersection number $ab'+a'b$. 

\begin{lemma}\label{genushyp}
  Let $C \subset \P^1 \times \P^1$ be a smooth projective curve of
  bidegree $(a,b)$ with $a,b >0$. Then $C$ is 
  irreducible, its genus equals $(a-1)(b-1)$, and $C$ is hyperelliptic
  if and only if $a\leq 2$ or $b\leq 2$.
\end{lemma}
\begin{proof}
From $a,b>0$ we find that $C$ is connected by \cite[Exc. III.5.6b]{H}.
Therefore, if $C$ were not irreducible, some components would
intersect in a singular point, contradicting smoothness of $C$. 
We conclude that $C$ is irreducible. 
Its genus equals $(a-1)(b-1)$ by \cite[Exc. III.5.6c]{H}. 
If $a\leq 2$ or $b\leq 2$, then projection of $C$ onto one of the two
factors of $\P^1 \times \P^1$ shows that $C$ is either isomorphic to
$\P^1$ or to a double cover of $\P^1$. In both cases $C$ is
hyperelliptic. If $\iota \colon\, C \rightarrow \P^1 \times \P^1$
denotes the embedding, then the canonical sheaf on $C$ is isomorphic to 
$\iota^* \O_{\P^1 \times \P^1}(a-2,b-2)$ \cite[Prop. II.8.20 and 
Exm. II.8.20.3]{H}. If $a,b>2$, then this is very ample, so $C$ is not
hyperelliptic \cite[Prop. IV.5.2]{H}. This finishes the proof.
\end{proof}

\begin{lemma}\label{ramgenus}
  Let $D$ be a smooth projective irreducible curve over an
  algebraically closed field of characteristic not equal to $2$, with
  genus $g(D)$ and function field $k(D)$.  Let $h \in k(D)$ be a
  rational function on $D$ and let $a$ denote the number of points on
  $D$ where $h$ has odd valuation. If $a>0$, then $k(D)[x]/(x^2-h)$ is
  a function field, corresponding to a smooth projective irreducible
  curve $C$ whose genus $g(C)$ equals $g(C) = 2g(D)-1+a/2$.
\end{lemma}
\begin{proof}
There is a point where $h$ has odd valuation, so $h$ is not a square
and $x^2-h$ is irreducible. It follows that $k(D)[x]/(x^2-h)$ is
a function field, corresponding to some smooth projective irreducible
curve $C$. The inclusion of function fields corresponds to a morphism 
$\varphi \colon \, C \to D$ of degree $2$, which is separable as the 
characteristic is not equal to $2$. The map $\varphi$ ramifies at all
points on $D$ where $h$ has odd valuation. For each such point $Q$ there
is a unique $P\in C$ with $\varphi(P)=Q$, at which the ramification 
index $e_P$ satisfies  
$2 \leq e_P \leq \deg \varphi=2$, so $e_P=2$. From the theorem of 
Riemann-Hurwitz \cite[Cor. IV.2.4]{H} we find 
$$
2g(C)-2 = \deg \varphi \cdot (2g(D)-2) + \sum_{P\in C} (e_P-1) = 
2(2g(D)-2) + a,
$$
from which we get $g(C) = 2g(D)-1+a/2$.
\end{proof}

The following lemma is no more than a reformulation that we will use
repeatedly. 

\begin{lemma}\label{tranint}
  Let $D \subset \A^2$ be a plane curve over an algebraically closed
  field, and $P$ a smooth point on $D$ corresponding with valuation
  $v_P$.  Let $h$ be a rational
  function on $\A^2$ that is regular on an open neighborhood $U\subset
  \A^2$ of $P$. Let $X\subset U$ be the vanishing locus of $h$ on
  $U$. Then $v_P(h) > 0$ if and only if $P$ is on $X$ and $v_P(h) =1$
  if and only if $X$ intersects $D$ transversally at $P$.
\end{lemma}
\begin{proof}
Let $\O_{\A,P}$ and $\O_{D,P}$ be the local rings of $P$ in $\A^2$ and
$D$ respectively. Since $\A^2$ is
smooth at $P$, the curve $D$ is locally principal at $P$, say given 
by $f=0$ with $f$ regular at $P$. Then there is an isomorphism 
$\O_{D,P} \isom \O_{\A,P}/(f)$ of local rings. The point $P$ lies on
$X$ if and only if $h$ is contained in the maximal ideal of
$\O_{\A,P}$, so if and only if $h$ is contained in the maximal of 
$\O_{D,P}$, i.e., $v_P(h)>0$. The intersection multiplicity of $D$ and
$X$ at $P$ is given by the length of $\O_{\A,P}/(f,h) \isom
\O_{D,P}/(h)$, which equals $v_P(h)$. By definition this intersection
is transversal when the multiplicity is $1$, so when  $v_P(h)=1$.
\end{proof}

\section{The standard model for the character varieties}
\label{equationsection}

For all integers $k,l$ with $l$ even, so that $J(k,l)$ is a knot, we
will define a model for the $\PSL_2(\C)$-character variety of $J(k,l)$
that is similar to the one often used in the literature. The following
polynomials will be useful.

\begin{definition}\label{newfs}
  Set $f_0=0$ and $f_1=1$. For all other $j\in \Z$, let $f_j \in
  \Z[u]$ be determined inductively (up and down) by the relation
  $f_{j+1}-uf_j+f_{j-1}=0$. For all integers $j$ we define $g_j$ by
  $g_j = f_j-f_{j-1}$.
\end{definition}

For notational convenience, we merge the sequences $(f_j)_j$ and
$(g_j)_j$ into a sequence $(\Phi_k)_k$ as follows.

\begin{definition}\label{phipsi}
  For each integer $j$ we define $\Phi_{2j}=f_j$ and
  $\Phi_{2j-1}=g_j$.  Furthermore, for each integer $k$ we set $\Psi_k
  = \Phi_{k+1}-\Phi_{k-1}$.
\end{definition}

\begin{lemma}\label{newbasicf}\label{fgat2}
  Let $j$ be any integer.  We have $f_{-j} = -f_j$ and $g_{-j} =
  g_{j+1}$.  If $j\neq 0$ then the polynomial $f_j$ has degree $|j|-1$ and is odd or
  even, based on the parity of its degree. The polynomial $g_j$ has
  degree $j-1$ for $j>0$ and degree $-j$ for $j\leq 0$.  We also have
  $f_j(2)=j$ and $g_j(2)=1$.
\end{lemma}
\begin{proof}
  The follows immediately by induction with respect to $j$, both
  upwards and downwards.
\end{proof} 

\begin{lemma}\label{stillbasic}
In the ring $\Z[u][s]/(s^2-us+1)\isom \Z[s,s^{-1}]$ 
we have $u = s+s^{-1}$ and $f_j = (s^j-s^{-j})/(s-s^{-1})$ and 
$g_j = (s^j+s^{1-j})/(s+1)$. We also have $f_{j-1}f_{j+1} = f_j^2-1$
and $g_jg_{j+1} = (u-2)f_j^2+1$.
\end{lemma}
\begin{proof}
  The expression for $f_j$ follows from induction, and the expression
  for $g_j$ follows immediately. The last equations are easily checked
  in terms of $s$.
\end{proof}

\begin{lemma}\label{phipsiprop}
  Let $k$ be any integer. Then we have $\Phi_{k+2}=u\Phi_k-\Phi_{k-2}$
  and $\Psi_{k+2}=u\Psi_k-\Psi_{k-2}$. We also have $\Phi_k =
  (-1)^{k+1} \Phi_{-k}$ and $\Psi_k = (-1)^{k+1} \Psi_{-k}$, while
  $\deg \Phi_k = \lfloor (|k|-1)/2\rfloor$.  Finally, we have
$$
\Psi_k = \left\{
\begin{array}{ll}
(u-2)\Phi_k = (u-2)f_j& \text{ if $k=2j$ is even,} \\
\Phi_k = g_j & \text{ if $k=2j-1$ is odd.}
\end{array}
\right.
$$
\end{lemma}
\begin{proof}
  The first statement follows from Definition \ref{newfs} and the
  second from Lemma \ref{newbasicf}. The last statement follows from
  Definition \ref{phipsi} and the identity $g_{j+1}-g_{j} = (u-2)f_j$,
  which is immediate from Definition \ref{newfs}.
\end{proof}

\begin{lemma}\label{tracewk}
  Suppose $A,B \in \SL_2(\C)$ satisfy $\tr A = \tr B$. Set $y = \tr
  A^2$ and $r = \tr A^{-1}B$. Let $k$ be any integer and set $m =
  \lfloor k/2 \rfloor$.  Define $W_{k} = (AB^{-1})^m(A^{-1}B)^m$ if
  $k$ is even and $W_{k} = (AB^{-1})^mAB(A^{-1}B)^m$ if $k$ is odd.
  Then we have
$$
\tr W_{k} = \Phi_{-k}(r)\Psi_{k}(r)(y-r)+2.
$$
%
%
\end{lemma}
\begin{proof}
  By Cayley-Hamilton we have $(\tr M)\cdot I = M+M^{-1}$ and $(\tr
  N)\cdot I = N+N^{-1}$ for all $M,N\in \SL_2(\C)$. Taking traces
  after multiplying the former equation by $N$ from the right and the
  latter by $M$ from the left, we obtain
\begin{equation}\label{traceformeq}
\tr \big(MN\big) = (\tr M)(\tr N) - \tr \big(M^{-1}N\big)
= (\tr M)(\tr N) - \tr \big(MN^{-1}\big)
\end{equation}
for all $M,N \in \SL_2(\C)$.
Set $c_{k,d} = \tr \big(W_k(A^{-1}B)^d\big)$ and 
$$
\gamma_{k,d} = \Phi_{-k}(r)\Psi_{k+2d}(r)(y-r)+f_{d+1}(r)-f_{d-1}(r).
$$
The Lemma is equivalent to the special case $d=0$ of the stronger
statement that $c_{k,d}=\gamma_{k,d}$ for all integers $k,d$. We will
prove by induction with respect to $k$ that this is true for $k$ and
all integers $d$.  We first use induction with respect to $d$ for $-1
\leq k \leq 1$.  we have $c_{0,0}=\tr I = 2=\gamma_{0,0}$ and $c_{0,1}
= \tr \big(A^{-1}B\big) = r=\gamma_{0,1}$ and $c_{1,-1}=\tr
\big(A^2\big)=y=\gamma_{1,-1}$.  Set $x=\tr A = \tr B$. Then by
(\ref{traceformeq}) we have $y=\tr \big(A^2\big) = (\tr A)^2-\tr I =
x^2-2$ and thus
$$
c_{1,0}=\tr W_1 = \tr \big(AB\big) = (\tr A)(\tr B)-\tr
\big(A^{-1}B\big) = x^2-r = y-r+2 =\gamma_{1,0}.
$$ 
We also have $c_{-1,0} =\tr \big(BA\big) = \tr \big(AB\big)
=\gamma_{1,0}=\gamma_{-1,0}$ and $c_{-1,1} = \tr \big(B^2\big)= (\tr
B)^2-\tr I = x^2-2=y$ by (\ref{traceformeq}).  Also by
(\ref{traceformeq}), we have
\begin{align}
  c_{k,d+1} &= \tr \big(W_k(A^{-1}B)^{d+1}\big) = \tr \big(W_k(A^{-1}B)^d(A^{-1}B)\big) \label{dind}\\
  &=\tr \big(W_k(A^{-1}B)^d\big)\big(\tr A^{-1}B\big) - \tr
  \big(W_k(A^{-1}B)^{d-1}\big) = rc_{k,d}-c_{k,d-1}.  \nonumber
\end{align}
The sequence $(\gamma_{k,d})_d$ satisfies the same recursion, so by
induction (increasing and decreasing) we find $c_{k,d}=\gamma_{k,d}$
for $-1\leq k \leq 1$ and all integers $d$. Therefore, we get
\begin{align*}
  c_{2,0}&=\tr \big(AB^{-1}A^{-1}B\big) = \big(\tr
  AB^{-1}A^{-1}\big)\big(\tr B\big) -
  \tr \big((AB^{-1}A^{-1})^{-1} B\big) \\
  &= \big(\tr (B^{-1})\big) \big(\tr B\big) - \tr \big(ABA^{-1}B\big)
  = x^2 - c_{1,1} = y+2-\gamma_{1,1} = \gamma_{2,0}.
\end{align*}
Together with $c_{2,-1} = \tr \big(AB^{-1}\big) = \tr
\big(B^{-1}A\big) = \tr\big((B^{-1}A)^{-1}\big) = \tr
\big(A^{-1}B\big)= r= \gamma_{2,-1}$ this is the basis for the
induction that shows $c_{2,d}= \gamma_{2,d}$ for all integers $d$, the
induction step following again from (\ref{dind}).  Now by
(\ref{traceformeq}) we have
\begin{align*}
  c_{k+2,d} &= \tr \big(W_{k+2}(A^{-1}B)^d) 
 = \tr \big((AB^{-1})W_{k} (A^{-1}B)^{d+1}\big) \\
  &=\big(\tr (AB^{-1})\big)\big(\tr (W_{k}(A^{-1}B)^{d+1})\big)
  -\tr \big((AB^{-1})^{-1} W_{k} (A^{-1}B)^{d+1}\big)\\
  &=rc_{k,d+1} - \tr\big( W_{k-2}(A^{-1}B)^{d+2}\big) = rc_{k,d+1} -
  c_{k-2,d+2}.
\end{align*}
From Lemma \ref{phipsiprop} it follows that we also have
$\gamma_{k+2,d} = r\gamma_{k,d+1} - \gamma_{k-2,d+2}$ for all integers
$k$ and $d$.  By induction with respect to $k$ it follows that
$c_{k,d}=\gamma_{k,d}$ for all integers $k$ and $d$.
\end{proof}

Analogous to \S \ref{charvars}, for any integer $k$, any
$\lambda_0\in \C^*$, and $r_0 \in \C$, we let $W_k(\lambda_0,r_0)$
denote the right-hand side of (\ref{wk}) with $A(\lambda_0)$ and
$B(\lambda_0,r_0)$ substituted for $a$ and $b$. Then for any integers
$k,n$, the assignments $a \mapsto A(\lam_0)$ and $b\mapsto
B(\lam_0,r_0)$ can be extended to a representation $\rho \in
\Hom(\pi_1(k,2n))$ if and only if we have $A(\lam_0)W_k(\lam_0,r_0)^n
= W_k(\lam_0,r_0)^nB(\lam_0,r_0)$, which results in four equations in
$\lambda_0$ and $r_0$. Again these equations reduce to a single
equation in $r_0$ and $\lambda_0$.

\begin{proposition}\label{eqFpropJ}
  Let $k,n$ be any integers.  Consider the ring
  $\Q[r,\lambda,\lambda^{-1}]$ and let $I$ denote the ideal generated
  by the four entries of the matrix $A(\lam)W_k(\lam,r)^n -
  W_k(\lam,r)^nB(\lam,r)$. Then $I$ is generated by
\begin{equation}\label{eqFJ}
F_{k,n}(\lam,r)= f_n(\tr W_k(\lambda,r)) \cdot F_{k,1}(\lam,r) - 
f_{n-1}(\tr W_k(\lam,r))
\end{equation}
with 
$$
F_{k,1}(\lam,r) = -\Phi_{-k}(r) \Phi_{k-1}(r) (y-r)+1,
$$
and with $y=\lambda^2+\lambda^{-2}$. 
\end{proposition}
\begin{proof}
  By Cayley-Hamilton we have $M^2 = tM - I$ for a matrix $M \in
  \SL_2(\C)$ with trace $t$; by induction, both up and down, we find
\begin{equation}\label{powerMeq}
M^j = f_j(t)\cdot M-f_{j-1}(t)\cdot I
\end{equation}
for all $j \in \Z$.  Completely analogous to Proposition
\ref{eqFprop}, we find $F_{k,n} = (\lam - \lam^{-1}) W_{12} + W_{22}$,
where $W_{ij}$ denotes the $(i,j)$-entry of $W_k(\lam,r)^n$. Let
$w_{ij}$ be the $(i,j)$-entry of $W_k(\lam,r)$ and set $t =
\tr(W_k)$. Then from (\ref{powerMeq}) we have $W_{12} =
f_{n}(t)w_{12}$ and $W_{22} = f_{n}(t)w_{22} - f_{n-1}(t)$, which
implies $F_{k,n}= f_n(t) F_{k,1} - f_{n-1}(t)$. From (\ref{powerMeq})
we also find 
$$
W_k(\lam,r) = (f_m(r)AB^{-1}-f_{m-1}(r) I)(f_m(r)A^{-1}B-f_{m-1}(r) I)
$$ 
if $k = 2m$ is even and
$$
W_k(\lam,r) = (f_m(r)AB^{-1}-f_{m-1}(r) I)AB(f_m(r)A^{-1}B-f_{m-1}(r) I)
$$
if $k = 2m+1$ is odd,
with $A= A(\lambda)$ and $B = B(\lambda,r)$.
From this one easily checks that 
$F_{k,1} = (\lam-\lam^{-1})w_{12}+w_{22}$ is indeed as given. 
\end{proof}

Recall that for all integers $k,l$, the $\SL_2(\C)$- and
$\PSL_2(\C)$-character varieties of the fundamental group $\pi_1(k,l)$
of the complement of $J(k,l)$ in $\Sp^3$ are denoted by $X(k,l)$ and
$Y(k,l)$ respectively.

\begin{proposition}\label{Ckl}
Let $k,l$ be any integers with $l$ even.
The variety $Y(k,l)$ is isomorphic to 
the subvariety $C(k,l)$ of $\A^2(r,y)$ defined by 
$$
C(k,l) \colon \,\,
f_n(t)\big(\Phi_{-k}(r) \Phi_{k-1}(r) (y-r)-1\big)+f_{n-1}(t)=0,
$$
with $t=\Phi_{-k}(r)\Psi_{k}(r)(y-r)+2$ and $n=l/2$. The variety 
$X(k,l)$ is isomorphic to the double cover of $C(k,l)$ defined in 
$\A^2(r,x)$ by $y=x^2-2$.
\end{proposition}
\begin{proof}
%
%
  Let $W_k(\lam,r)$ be as in Proposition \ref{eqFpropJ}. Then $t=\tr
  W_k(\lam,r)$ by Lemma \ref{tracewk}.  We conclude that $C(k,l)$ is
  the curve given by $F_{k,n}=0$ (in terms of $r$ and $y$) in
  $\A^2(r,y)$.  Completely analogous to \S \ref{charvars}, the
  varieties $X(k,l)$ and $Y(k,l)$ have models in $\A^2(r,x)$ and
  $\A^2(r,y)$ given by $F_{k,n}=0$ in terms of $r$ and $x$ and in
  terms of $r$ and $y$ respectively. The proposition follows.
\end{proof}

Note that if $kl=0$, then the variety $C(k,l)$ is empty. This reflects
the fact that in those cases $J(k,l)$ is the trivial knot, so
$\pi_1(k,l)$ is a free abelian group, which has no nonabelian
representations. The following lemma will be useful later.

\begin{lemma}\label{specpointsoldboth}
Suppose $k,l$ are integers with $l$ even.
If $P=(r_0,y_0) \in C(k,l)(\Qbar)$ is a point with 
$\Psi_k(r_0) =0$, then $k$ is even and $P=(2,2-\frac{4}{kl})$. 
\end{lemma}
\begin{proof}
  By assumption the variety $C(k,l)$ is not empty, so we conclude
  $kl\neq 0$.  Set $n=l/2$ and $t_0 =
  \Phi_{-k}(r_0)\Psi_{k}(r_0)(y_0-r_0)+2$. Then by Proposition
  \ref{Ckl} we have
\begin{equation}\label{CklatP}
  f_n(t_0)\big(\Phi_{-k}(r_0) \Phi_{k-1}(r_0) 
(y_0-r_0)-1\big)+f_{n-1}(t_0)=0.
\end{equation}
From $\Psi_k(r_0) =0$ we get $t_0=2$, and by Lemma \ref{fgat2} we have
$f_n(t_0)=n$ and $f_{n-1}(t_0) = n-1$.  Suppose we had
$\Phi_{-k}(r_0)=0$. Then the left-hand side of (\ref{CklatP}) equals
$-n+(n-1)=-1$. From this contradiction we conclude $\Phi_{-k}(r_0)\neq
0$. If $k$ were odd then we would have $0=\Psi_k(r_0) =
\Phi_{-k}(r_0)\neq 0$ by Lemma \ref{phipsiprop}, so we conclude that
$k$ is even and find $0 = \Psi_k(r_0) = (2-r_0)\Phi_{-k}(r_0)$ by
Lemma \ref{phipsiprop}.  This implies $r_0=2$. By Lemmas \ref{fgat2}
and \ref{phipsiprop} we then have $\Phi_{-k}(r_0)=-\frac{1}{2}k$ and
$\Phi_{k-1}(r_0)=1$, so the left-hand side of (\ref{CklatP}) equals
$n(-\frac{1}{2}k(y_0-2)-1)+n-1$.  Solving (\ref{CklatP}) for $y_0$
gives $y_0=2-\frac{2}{kn}= 2-\frac{4}{kl}$.
\end{proof}

The models of $X(k,l)$ and $Y(k,l)$ described in Proposition 
\ref{Ckl}, up to perhaps a linear transformation, are the 
standard models. Their usual projective closures in $\P^2$ and 
$\P^1 \times \P^1$ are highly singular. 
Note that the trace $\tr(W_k)$ is linear in $y$ for all 
nonzero integers $k$. We can exploit this to give 
a model of $Y(k,l)$ with a smooth completion in $\P^1 \times \P^1$.
This will be done in the next section.

\section{A new model for the character varieties}
\label{newmodelsection}

In this section we introduce a new model for $Y(k,l)$. It 
does not respect integrality, but is geometrically nicer than 
the standard model in the sense that it is projective and 
all its irreducible components are smooth. 
%
%
The coordinates $r$ and $y$ from the previous section are the trace
functions $t_{a^{-1}b}$ and $t_{a^2}$ respectively. For the new model
we will replace $y$ by the trace function $t=t_{W_k}$, which is linear
in $y$.

Let $D(k,l)$ be the variety in $\P^1_{\Q}(r)\times \P^1_{\Q}(t)$ that
is the projective closure of the affine variety given by 
\begin{equation}\label{defeqsnew}
\Phi_{k+1}(r)\Phi_{l-1}(t) = \Phi_{k-1}(r)\Phi_{l+1}(t).
\end{equation}
Note that for $l=2n$, expressed in terms of the polynomials $f_j$ and
$g_j$ this is 
$$
\begin{array}{ll}
g_{m+1}(r)g_n(t)= g_m(r)g_{n+1}(t) & \mbox{if } k=2m, l=2n\\
f_{m+1}(r)g_n(t)= f_m(r)g_{n+1}(t) & \mbox{if } k=2m+1, l=2n. 
\end{array}
$$
Subtracting $\Phi_{k-1}(r)\Phi_{l-1}(t)$ from both sides of 
(\ref{defeqsnew}), and using Definition \ref{phipsi},
we find that $D(k,l)$ is also given by the alternate equations
\begin{equation}\label{defeqsalt}
\Psi_k(r)\Phi_{l-1}(t) = \Phi_{k-1}(r)\Psi_l(t).
%
\end{equation}

\begin{remark}\label{bidegrees}
  Let $k,n$ be any integers.  For $k=n=0$, the variety $D(k,2n)$ is
  the full $\P^1 \times \P^1$, while for $k = \pm 1$ and $n \in
  \{0,k\}$, it is in fact empty.  Suppose we are not in any of those
  cases. Then $D(k,2n)$ has dimension $1$, and from Lemma
  \ref{phipsiprop} one quickly finds the bidegree of $D(k,2n)$.  It
  equals $(\lfloor |k|/2 \rfloor, |n|)$ when $k\neq \pm 1$.  For $k =
  \pm 1$ the bidegree equals $(0,kn-1)$ if $kn>0$ and it equals
  $(0,-kn)$ if $kn < 0$.
\end{remark}

In some sense it seems natural to include the line given by $t =
\infty$ in $D(k,2n)$ when $k=\pm 1$ and $kn>0$; then $D(k,2n)$ would
be a curve of bidegree $(\lfloor |k|/2 \rfloor, |n|)$, as long as this
differs from $(0,0)$. Doing this is also natural in the sense that it
would follow from a slightly different definition for $D(k,2n)$ that
gives an explicit equation on an affine chart that includes the line
$t = \infty$.  We have chosen not to do this in order to keep
$D(k,2n)$ birationally equivalent with the standard model $C(k,2n)$
for $Y(k,2n)$. Before we prove this, we state a few lemmas.

\begin{lemma}\label{unitideal}
For every $j \in \Z$ 
the ideals $(g_j,g_{j-1})$, $(f_j,f_{j-1})$, 
$(\Phi_{j+1},\Phi_{j-1})$, and $(\Psi_j,\Phi_{j-1})$ 
of $\Z[u]$ all equal the unit ideal. 
\end{lemma}
\begin{proof}
From the identity $1=f_{j-1}g_{j-1}-f_{j-2}g_j$ we find that 
the first ideal is the unit ideal. The identities in Lemma 
\ref{stillbasic} show that the second ideal and the ideal 
$((u-2)f_i,g_i)=(\Psi_{2i},\Phi_{2i-1})$ are unit ideals
for any integer $i$. 
From $g_{i+1}=f_{i+1}-f_{i}$ it follows that 
$(\Psi_{2i+1},\Phi_{2i}) = (g_{i+1},f_{i}) = 
(f_{i+1},f_i)=(1)$. This proves that the last ideal is the 
unit ideal both when $j$ is odd and when $j$ is even. 
The third ideal is of the form of the first or second ideal, 
depending on the parity of $j$, so it is also the unit ideal.
\end{proof}

\begin{lemma}\label{specpointsboth}
Let $k,l$ be any integers and $P=(r_0,t_0)$ a $\Qbar$-point on 
the standard affine part of $\P^1 \times \P^1$.
Then the following statements are equivalent. 
\begin{enumerate}
\item We have $\Psi_k(r_0) = \Psi_l(t_0) = 0$.
\item The point $P$ lies on $D(k,l)$ and $\Psi_k(r_0) = 0$. 
\item The point $P$ lies on $D(k,l)$ and $\Psi_l(t_0) = 0$.
\end{enumerate}
\end{lemma}
\begin{proof}
  To show equivalence of (1) and (2), assume we have $\Psi_k(r_0) =
  0$.  From Lemma \ref{unitideal} we conclude $\Phi_{k-1}(r_0) \neq
  0$, so (\ref{defeqsalt}) shows that $P$ lies on $D(k,l)$ if and only
  if $\Psi_l(t_0) = 0$.  Equivalence of (1) and (3) follows by
  symmetry.
\end{proof}

\begin{proposition}\label{charvar}
Suppose $k,l$ are integers with $l$ even and $kl\neq 0$.
The map $\A^2(r,y) \rightarrow \P^1(r)\times \P^1(t)$ that sends 
$(r,y)$ to $(r,\tr(W_k))$, with $\tr(W_k)$ as in Lemma \ref{tracewk}, 
induces a birational morphism from $C(k,l)$ to $D(k,l)$.
\end{proposition}
\begin{proof}
  Let $\sigma$ denote the map described. It is clearly well defined
  everywhere and therefore induces a morphism from $C(k,l)$ to its
  image. By Lemma \ref{tracewk}, the map $\sigma$ is given by $(r,y)
  \mapsto (r,\Phi_{-k}(r)\Psi_k(r)(y-r)+2)$, which has a birational
  inverse, given by $(r,t) \mapsto
  (r,r+(t-2)\Phi_{-k}(r)^{-1}\Psi_k(r)^{-1})$.  Note that $\Phi_{-k}$
  divides $\Psi_k$, so $\sigma$ induces an isomorphism from the open
  subset $U$ of $\A^2(r,y)$ given by $\Psi_k(r) \neq 0$ to the open
  subset $V$ of the standard affine part of $\P^1(r) \times \P^1(t)$
  given by $\Psi_k(r) \neq 0$.  These open sets are dense because
  $\Psi_k \neq 0$ for $k\neq 0$.  Set $n=l/2$.  By Proposition
  \ref{eqFpropJ} the image $\sigma(C(k,l))$ is on $V$ given by
$$
f_n(t)\left( \frac{(t-2)\Phi_{k-1}(r)}{\Psi_k(r)}-1\right) +
f_{n-1}(t)=0,
$$
which is equivalent to the equation for $D(k,l)$ in (\ref{defeqsalt})
by Lemma \ref{phipsiprop}. Therefore $U\cap C(k,l)$ is isomorphic with
$V \cap D(k,l)$. Since $l\neq 0$, there are only finitely many $t_0$
with $\Psi_l(t_0)=0$. Therefore, by Lemmas \ref{specpointsoldboth} and
\ref{specpointsboth}, the curves $C(k,l)$ and $D(k,l)$ contain no full
components outside $U$ and $V$ respectively, so they are isomorphic
outside a finite number of points, and therefore birationally
equivalent.
\end{proof}


\begin{remark}\label{lines}
  We have already seen that $Y(k,l)$ is empty if $kl=0$.  Suppose
  $|k|=1$ and $l\not \in \{0,2k\}$ or suppose $k=l\in \{\pm 2\}$.
  Then $D(k,l)$ consists of a finite number of lines (cf. Remark
  \ref{bidegrees}).  By Proposition \ref{charvar} this implies that
  $C(k,l)$ and $Y(k,l)$ consist of a number of curves of genus
  $0$. The corresponding knots $J(k,l)$ are not hyperbolic in all
  these cases and we will not give them much further attention.
\end{remark}

Note that from Lemma \ref{phipsiprop} it follows that $D(k,l)$ and
$D(-k,-l)$ are the same, reflecting the fact that $J(-k,-l)$ is the
mirror image of $J(k,l)$.

The symmetry of the equation for $D(k,l)$ in (\ref{defeqsnew}) shows
that the automorphism of $\P^1 \times \P^1$ that sends $(r,t)$ to
$(t,r)$, induces an isomorphism from $D(k,l)$ to $D(k,l)$. Since $r$
and $t$ are the traces of the elements $\beta$ and $\alpha$ in the 
fundamental
group of $\Sp^3\setminus J(k,l)$
respectively, as described in Figure \ref{newfancypicture}, it follows
from the discussion at the end of \S \ref{famtwobridge} that this
isomorphism is induced by turning upside down the $4$-plat
representation as in Figure \ref{fourplatgen}, which also switches
$\alpha$ and $\beta$.  In particular this
applies when $k=l\neq 0$, in which case $D(l,l)$ contains an
irreducible component given by $r=t$. This means that $Y(l,l)$ is
reducible for $|l|>2$.  The reducibility of $Y(l,l)$ for $|l|>2$ was
already known from \cite{Ohtsuki}, \cite{Ri2}, as for the associated
two-bridge knot $K(p,q)$ we have $q^2 \equiv 1 \pmod p$. We can now
identify the component given by $r=t$.

\begin{proposition}\label{Yzero}
  Suppose $l$ is an even integer and $|l|>2$.  Then under the 
  birational equivalence between
  $Y(l,l)$ and $D(l,l)$, the irreducible component $Y_0(l,l)$
  corresponds to the line given by $r=t$.
\end{proposition}
\begin{proof}
  The automorphism of $D(l,l)$ that sends $(r,t)$ to $(t,r)$ is
  induced by turning upside down the $4$-plat presentation in Figure
  \ref{fourplatgen}. By \cite[proof of Prop. 5.5]{Ohtsuki},  
  this involution acts trivially on the component
  $Y_0(l,l)$ of $Y(l,l)$. This implies that $Y_0(l,l)$ corresponds to
  the line given by $r=t$.
\end{proof}

\begin{proof}[Proof of Theorem \ref{mainthree}]
Let $\rho \colon \pi_1 \big( \Sp^3\setminus J(k,l)\big) \to 
\SL_2(\C)$ denote a lift of the discrete faithful representation 
(cf. end of \S \ref{sltwo}). 
By definition the trace field $F\big(J(k,l)\big)$ of $J(k,l)$ is 
generated by the 
traces of the elements in the image of $\rho$, so it equals 
the field of definition of the point $\chi$ on $X(k,l)$ associated to 
$\rho$. 
The images of meridians under $\rho$ are parabolic 
(this follows from 
\cite[Ch. 5]{T}, cf. \cite[\S 2]{Ri2} and \cite[\S 1]{rileytwo}),
so their traces equal $\pm 2$. Therefore, in terms of the coordinates 
$r,x$ as in Proposition \ref{Ckl}, the point $\chi$ 
satisfies $x=\pm 2$ and maps to the point $(r_0,2)$ on 
$C(k,l)\subset \A^2(r,y)$ for some $r_0\in \C$. The trace field then 
equals $\Q(r_0)$. Substituting $y=2$ in the 
equation for $C(k,l)$ gives a polynomial with root $r_0$ of 
degree $-kl/2$ if $kl<0$ 
and degree $kl/2 - 1$ if $kl>0$. This proves the first two bounds. 
If $k=l$, then the canonical component of $Y(l,l)$ corresponds 
by Propositions \ref{charvar} and \ref{Yzero} to the component 
of $C(k,l)$ given by $r = \Phi_{-k}(r)\Psi_{k}(r)(y-r)+2$. 
Substituting $y=2$ and taking out a factor $r-2$ gives an equation 
of degree $|l|-1$, which proves the final upper bound.
The first two bounds also follow immediately from \cite[\S 3]{rileytwo}. 
\end{proof}

Based on Proposition \ref{Yzero}, we give the following definition.

\begin{definition}\label{zeros}
  For each nonzero even integer $l$, let $D_0(l,l)$ denote the
  component of $D(l,l)$ given by $r=t$ and let $D_1(l,l)$ denote the
  projective closure of the scheme-theoretic complement of $D_0(l,l)$
  in $D(l,l)$; if $|l|>2$, then we denote the scheme-theoretic
  complement of $Y_0(l,l)$ in $Y(l,l)$ by $Y_1(l,l)$ and the
  scheme-theoretic complement of $X_0(l,l)$ in $X(l,l)$ by $X_1(l,l)$.
\end{definition}

Note that $D_1(2n,2n)$ is given by
$(g_{n+1}(r)g_n(t)-g_n(r)g_{n+1}(t))/(t-r)=0$ for any nonzero integer
$n$.  For $|n|=1$ (the trefoils, which are nonhyperbolic), we see that
$D_1(2n,2n)$ is empty; for $|n|>1$ it is of bidegree $(|n|-1,|n|-1)$.

\section{Smoothness and Irreducibility of the character varieties}
\label{smoothirred}

In this section we will prove the following theorem, covering all
hyperbolic knots of the form $J(k,l)$.

\begin{theorem}\label{thmDsmooth}
Let $l$ be an even integer with $|l|\geq 2$.
If $k$ is an integer with $k\neq l$ and $|k|\geq 2$,
then $D(k,l)$ is smooth over $\Q$.
If $|l|>2$, then $D_1(l,l)$ is smooth over $\Q$.
\end{theorem}

We split the proof of the first part of Theorem \ref{thmDsmooth} into
three cases, based on the parity of $k$ and the sign of $kl$ in case
$k$ is even.  Theorem \ref{thmDsmooth} will be proved at the end of
this section as a corollary of Propositions \ref{oppsigns},
\ref{proptwokponetwon}, \ref{proptwoktwon}, and \ref{Donesmooth}.  The
approach is the same for all cases, but the details are different.  We
first sketch the idea behind our approach.

\begin{definition}
For each integer $k$ we set $h_k = \Phi_{k+1}/\Phi_{k-1}$.
\end{definition}

Suppose $k,l$ are integers and $P=(r_0,t_0)$ is a singular point on
the affine part of $D(k,l)$.  We show that this implies
$\Phi_{k-1}(r_0)\neq 0$ and $\Phi_{l-1}(t_0)\neq 0$. Then $D(k,l)$ can
be given around $P$ by $h_k(r)=h_l(t)$.  The fact that $P$ is a
singular point is then equivalent with the fact that $r_0$ and $t_0$
are critical points for $h_k$ and $h_l$ respectively.  We show that
for each $k$, the values of $h_k$ at its critical points are all
different from each other, and they are also different from the values
of $h_l$ at all its critical points when $k\neq l$.  This is done
using complex absolute values or $p$-adic valuations, depending on the
case.  The equation $h_k(r_0)=h_l(t_0)$ then implies $k=l$ and
$r_0=t_0$.  Indeed, for $k=l$ the component $D_0(l,l)$ given by $r=t$
intersects the curve $D_1(l,l)$ in singular points of $D(l,l)$.

\begin{definition}
For every $n \in \Z$, set $F_n = f_{n+1}'f_n-f_{n+1}f_n'$ and 
$G_n = g_{n+1}'g_n-g_{n+1}g_n'$.
\end{definition}

Note that $F_n$ and $G_n$ are the numerators 
of the derivatives of $f_{n+1}/f_n$ and $h_{2n}$.
We first state some facts. 

\begin{lemma}\label{irredstats}
For every $n\in \Z$ the following statements hold. 
\begin{enumerate}
\item If $n\neq 0$, then the polynomial $f_n$ is separable. 
\item The polynomials $F_n$ and $G_n$ have leading coefficient $\pm 1$. 
\item We have $(u+2)G_n = f_{2n} + 2n =
  \frac{s^{2n}-s^{-2n}}{s-s^{-1}}+2n$ in $\Z[u][s]/(s^2-us+1)$. 
\item We have $G_n(2)=n$ and $G_n(-2) = \frac{1}{3}n(4n^2-1)$.
\item For any field $\F$ with characteristic not dividing $2n-1$, the
  polynomial $g_n$ is separable over $\F$ and we have $(G_n,g_n)=(1)$
  in $\F[u]$.
\end{enumerate}
\end{lemma}
\begin{proof}
Set $h = (s^{n+1}-s^{n-1})f_n \in \Z[u][s]/(s^2-us+1)\isom
\Z[s,s^{-1}]$. 
Then we have $h=s^{2n}-1$, 
which is separable, as $s\frac{dh}{ds} - 2nh = 2n$ is a nonzero constant
for $n\neq 0$. We conclude that $f_n$ does not have multiple factors 
either, which proves (1). 

The polynomials $g_n$ and $g_{n+1}$ are monic, while their degrees
differ by $1$.  This implies that the leading terms of $g_{n+1}'g_n$
and $g_n'g_{n+1}$ also differ by $1$. Therefore, their difference
$G_n$ indeed has leading coefficient $\pm 1$. The same argument
applies to $F_n$, which proves (2).

The identity in (3) is easily verified in $\Z[s,s^{-1}]$.  Note that
we have $g_n' = \frac{dg_n}{ds} /\frac{du}{ds} = \frac{dg_n}{ds} \cdot
\frac{s^2}{s^2-1}$.

One can prove (4) by dividing the identity of (3) by
$u+2=s^{-1}(s+1)^2$, setting $s = \pm 1$ and applying l'H\^opital's
rule. Alternatively, it follows from induction that we have $g_n(-2) =
(-1)^{n-1} (2n-1)$, while by Lemma \ref{fgat2} we have $g_n(2) =1$.
From $g_{n+1}' = [ug_n-g_{n-1}]' = ug_n'-g_{n-1}'+g_n$ we then find by
induction that $g_n'(2) = \frac{1}{2}n(n-1)$, while we have
$g_{n+1}'(-2) = (-1)^n\frac{1}{6}n(n-1)(2n-1)$.  It follows that
$G_n(\pm 2)$ is as given.

For (5), let $\F$ be a field with characteristic not dividing $2n-1$.
By Lemma \ref{stillbasic} we have $(s+1)s^{n-1}g_n = s^{2n-1}+1$ and
the reduction of this polynomial to $\F$ is separable.  Then the
reduction of the polynomial $g_n$ has no multiple factors either, so
$g_n$ is separable over $\F$.  The ideal $(G_n,g_n) \subset \F[u]$
contains $g_{n+1}'g_n-G_n = g_{n+1}g_n'$.  By Lemma \ref{unitideal},
the polynomials $g_n$ and $g_{n+1}$ have no roots in common, and as
$g_n$ is separable over $\F$, it also has no roots in common with
$g_n'$, so in $\F[u]$ we find $(1)= (g_n,g_{n+1}g_n') = (G_n,g_n)$,
which finishes the proof of (5).
\end{proof}

For each integer $k$, set $\Delta_k =
\Phi_{k+1}'\Phi_{k-1}-\Phi_{k+1}\Phi_{k-1}'$.  Note that for even $k$,
say $k=2m$, we have $\Delta_k=G_m$, while for odd $k$, say $k=2m+1$,
we have $\Delta_k=F_m$.

\begin{lemma}\label{firststep}
Let $k$ and $l$ be any integers with $l$ even.
Suppose $P=(r_0,t_0)$ is a singular $\Qbar$-point of the standard affine 
part of $D(k,l)$. Then 
we have $\Phi_{k-1}(r_0)\neq 0 \neq \Phi_{l-1}(t_0)$ and 
$\Delta_k(r_0)=\Delta_l(t_0)=0$.
\end{lemma}
\begin{proof}
  Set $F=\Phi_{k+1}(r)\Phi_{l-1}(t)-\Phi_{k-1}(r)\Phi_{l+1}(t)$ and
  $F_x=\partial F/\partial x$ for $x=r,t$.  Then we have $F(P) =
  F_r(P) = F_t(P)=0$, so also
  \[ 0=\Phi_{l-1}'(t_0)F(P)-\Phi_{l-1}(t_0)F_t(P)=\Phi_{k-1}(r_0)\Delta_l(t_0)\]
  and
  \[0=\Phi_{l+1}'(t_0)F(P)-\Phi_{l+1}(t_0)F_t(P)=\Phi_{k+1}(r_0)\Delta_l(t_0).\]
  By Lemma \ref{unitideal} we can not have
  $\Phi_{k-1}(r_0)=\Phi_{k+1}(r_0)=0$, so we have $\Delta_l(t_0)=0$
  and, similarly, $\Delta_k(r_0)=0$.  Since $l$ is even, say $l=2n$,
  we have $\Delta_l=G_n$.  From Lemma \ref{irredstats}(5) we conclude
  $\Phi_{l-1}(t_0)=g_n(t_0)\neq 0$. If we had $\Phi_{k-1}(r_0)=0$,
  then $F(P)=0$ would imply $\Phi_{k+1}(r_0)=0$, which contradicts
  Lemma \ref{unitideal}. We conclude $\Phi_{k-1}(r_0)\neq 0$.
\end{proof}

The following lemma will be used to prove smoothness at infinity.

\begin{lemma}\label{inftysmooth}
  Let $e,f \in \Z[r]$ and $g,h \in \Z[t]$ be nonzero separable
  polynomials, and assume that $\deg e - \deg f = \pm 1$ and $\deg g -
  \deg h = \pm 1$.  Let $C \subset \P^1(r)\times \P^1(t)$ be the
  projective closure of the affine curve given by $e(r)g(t) =
  f(r)h(t)$. Then $C$ is smooth at its points at infinity and the two
  lines at infinity intersect $C$ transversally everywhere.
\end{lemma}
\begin{proof}
  Set $r'=r^{-1}$ and $t'=t^{-1}$ in the function field $\Qbar(r,t)$
  of $\P^1\times \P^1$ over $\Qbar$. Let $L$ be the line at infinity
  given by $r'=0$.  By symmetry between $r$ and $t$ it suffices to
  consider the points in $L \cap C$.  This means it suffices to check
  all points on $C$ with $r'=0$ in the affine patches with coordinates
  $(r',t)$ and $(r',t')$. Set $a=\deg e$ and $b=\deg g$. By symmetry
  between $(e,g)$ and $(f,h)$ we may assume $\deg f = a+1$. Set
  $e'(r') = r'^{\deg e} e(1/r')$ and define $f',g',h'$ similarly. Note
  that $e',f',g',h'$ do not vanish at $0$.  Then on the affine patch
  with coordinates $(r',t)$, the curve $C$ is given by $r'e'(r')g(t) =
  f'(r')h(t)$. Now first consider the case $\deg h = b+1$. Then $C$ is
  of bidegree $(a+1,b+1)$.  The line $L$ is of bidegree $(1,0)$, so
  the intersection number $L\cdot C$ equals $b+1$, when counting the
  intersection points with multiplicities. For each root $\tau$ of
  $h(t)$ there is a point $(r',t)=(0,\tau)$ on $L \cap C$, so there
  are at least $b+1$ different points on $L\cap C$. This implies that
  all intersection multiplicities are $1$, which shows that all points
  on $L\cap C$ are nonsingular and all intersections are transversal.
  Now consider the case $\deg h = b-1$. Then $C$ is of bidegree
  $(a+1,b)$, so we have $L\cdot C = b$. On the patch with coordinates
  $(r',t')$, the curve $C$ is given by $r'e'(r')g'(t') =
  t'f'(r')h'(t')$.  Now if $h(0)\neq 0$, then $\deg h'(t') = b-1$, and
  for each of the $b$ roots $\tau$ of $t'h'(t')$ there is a point
  $(r',t')=(0,\tau)$ on $L\cap C$.  If $h(0)=0$, then $h$ has a simple
  root at $0$ as $h$ is separable, so $\deg h'(t') = b-2$ and there
  are also $b$ points on $L\cap C$, namely $(r',t)=(0,0)$ and the
  $b-1$ points $(0,\tau)$ for any root $\tau$ of $t'h'(t')$. In either
  case we find that all intersection multiplicities are $1$, so all
  points on $L\cap C$ are nonsingular and the intersections are
  transversal.
\end{proof}

In the case that $k$ is even and $kl$ is negative we use the
following lemma. Recall that we have $h_{2n} = \Phi_{2n+1}/\Phi_{2n-1}
= g_{n+1}/g_n$.

\begin{lemma}\label{complexabs}
  Let $n$ be any nonzero integer, and $\omega \in \C$ a root of $G_n$.
  If $n>0$, then $|h_{2n}(\omega)| > 1$, and if $n<0$, then
  $|h_{2n}(\omega)|<1$.
\end{lemma}
\begin{proof}
  Note that $h_{2n}(\omega)$ is well defined, as $g_n(\omega)\neq 0$
  by Lemma \ref{irredstats}(5). Assume $n>0$, and choose a $\sigma \in
  \C^*$ such that $\omega = \sigma + \sigma^{-1}$. Then from Lemma
  \ref{irredstats}(3) we find $\sigma^{2n}-\sigma^{-2n} = -2n(\sigma -
  \sigma^{-1})$, which shows that $\sigma^{2n}-\sigma^{-2n}$ and
  $\sigma - \sigma^{-1}$ are in opposite half-planes (upper and lower
  half-plane, both including the real line).  Note that for each $z\in
  \C^*$, the values of $z, z-z^{-1}$, and $z-\overline{z}$ are all in
  the same half-plane, so we conclude that
  $\sigma^{2n}-\overline{\sigma}^{2n}$ and $\sigma -
  \overline{\sigma}$ are in opposite half-planes. Since both these
  values are purely imaginary, we conclude $(\sigma -
  \overline{\sigma})(\sigma^{2n}-\overline{\sigma}^{2n})\geq 0$, with
  equality if and only if $\sigma^{2n}$ is real. Set $\alpha =
  \sigma\overline{\sigma} = |\sigma|^2>0$. Then $\alpha$ and
  $\alpha^{2n}$ either both exceed $1$, or they both do not, and we
  have $(\alpha-1)(\alpha^{2n}-1) \geq 0$ in either case, with
  equality if and only if $\alpha=1$. Now we have
\begin{align}
|\sigma^{2n+1}+1|^2-&|\sigma^{2n}+\sigma|^2 = 
   (\sigma^{2n+1}+1)(\overline{\sigma}^{2n+1}+1) - 
   (\sigma^{2n}+\sigma)(\overline{\sigma}^{2n}+\overline{\sigma}) \cr
& = (\alpha-1)(\alpha^{2n}-1) + 
    (\sigma - \overline{\sigma})(\sigma^{2n}-\overline{\sigma}^{2n})\geq 0,
    \label{absineq}
\end{align}
with equality if and only if $|\sigma|^2=\alpha=1$ and $\sigma^{2n}$
is real, so if and only if $\sigma^{2n} = \pm 1$. If $\sigma^{2n} =
\pm 1$, then from $\sigma^{2n}-\sigma^{-2n} = -2n(\sigma -
\sigma^{-1})$ we find $\sigma = \sigma^{-1}$, so $\sigma = \pm 1$, and
$\omega = \pm 2$. From $G_n(2) = n$ and $G_n(-2) =
\frac{1}{3}n(4n^2-1)$ (see Lemma \ref{irredstats}) we conclude that
the inequality in (\ref{absineq}) is strict, and $|\sigma^{2n+1}+1| >
|\sigma^{2n}+\sigma|$. As we have $g_n(\omega) =
(\sigma^n+\sigma^{1-n})/(\sigma+1)$, we get
$$
|h_{2n}(\omega)| =
\left|\frac{\sigma^{2n+1}+1}{\sigma^{2n}+\sigma}\right| > 1.
$$
The proof for $n<0$ is similar. In that case $\sigma -
\overline{\sigma}$ and $\sigma^{2n}-\overline{\sigma}^{2n}$ are in the
same half-planes, and $(\alpha-1)(\alpha^{2n}-1) \leq 0$.
\end{proof}

We now have all tools to handle the case that $k$ is even and 
$kl$ is negative. This is done in the following proposition.

\begin{proposition}\label{oppsigns}
  Let $k,l$ be any even integers with $kl<0$. Then $D(k,l)$ is smooth
  over $\Q$.
\end{proposition}
\begin{proof}
  Set $m=k/2$ and $n=l/2$.  The curve $D(k,l)$ is the same as
  $D(-k,-l)$, so without loss of generality we assume $l>0$ and
  $k<0$. We will argue over $\C$.  Assume $P=(r_0,t_0)$ is a singular
  point of the standard affine part of $D(k,l)$ with $r_0,t_0 \in \C$.
%
%
  By Lemma \ref{firststep} we have $g_m(r_0)\neq 0 \neq g_n(t_0)$, so
  we may rewrite $F(P)=0$ as $h_k(r_0) = h_l(t_0)$.  This contradicts
  the fact that from Lemma \ref{complexabs} we have $|h_k(r_0)| < 1 <
  |h_l(t_0)|$, so there is no singular point on the affine part of
  $D(k,l)$. The points at infinity are smooth by Lemma
  \ref{inftysmooth}.
\end{proof}

We will see that in the remaining cases ($k$ is odd or $kl$ is
positive) we can use non-archimedean places instead of complex
absolute values.  We use the following lemmas.

\begin{lemma}\label{lemmirrednewids}
For every $n\in \Z$, we have the following identities
\begin{align}
2-u&= g_{n+1}^2+g_n^2 - ug_ng_{n+1}, \label{start}\\
(4-u^2)G_n &= (2n+1)g_n^2+(2n-1)g_{n+1}^2-2nug_ng_{n+1}, \label{surprise}\\
(4-u^2)G_n &= g_n^2-g_{n+1}^2 - 2n(u-2), \label{twontminn}\\
(u^2-4)F_n &= f_{n+1}^2-f_n^2-(2n+1)\label{Feq}.
\end{align}
\end{lemma}
\begin{proof}
  All these identities can be verified in $\Z[u][s]/(s^2-us+1) \isom
  \Z[s,s^{-1}]$.  Note that we have $g_n' = \frac{dg_n}{ds}
  /\frac{du}{ds} = \frac{dg_n}{ds} \cdot \frac{s^2}{s^2-1}$, and
  something similar for $f_n'$.  Equation (\ref{start}) also follows
  from the last equation of Lemma \ref{stillbasic} and the relation
  $tf_n = f_{n+1}+f_{n-1}$.  Equation (\ref{twontminn}) also follows
  by subtracting $2n$ times the equation (\ref{start}) from
  (\ref{surprise}).
\end{proof}

It turns out that for the non-archimedean places it is more useful to
look at the values of $h_l^2-1$ than those of $h_l$, which we used in
the case that $k$ is even and $kl$ is negative.  For any integer $n$
and any root $\omega$ of $G_n$ we have $g_n(\omega) \neq 0$ by Lemma
\ref{irredstats}(5); dividing equation (\ref{twontminn}) by
$g_n(\omega)^2$, we get
\begin{equation}\label{sqminoneeq}
h_{2n}(\omega)^2-1 = \left(\frac{g_{n+1}(\omega)}{g_n(\omega)}\right)^2-1 
= \frac{2n(2-\omega)}{g_n(\omega)^2}.
\end{equation}

Recall from \S \ref{algebra} that for any prime $p$, the discrete
valuation on $\Q$ associated to $p$ is denoted by $v_p$ and satisfies
$v_p(p)=1$.  We scale each discrete valuation $v$ on any number field
so that it restricts to $v_p$ on $\Q$ for some prime $p$, i.e., such
that $v(p)=1$.

\begin{lemma}\label{vgntauzero}
  Let $n$ be any integer, and $p$ a prime dividing $2n$.  Let $K$ be a
  number field containing a root $\omega$ of $G_n$. Let $v$ be a
  valuation on $K$ with $v(p)=1$. Then $v(g_n(\omega))=0$.
\end{lemma}
\begin{proof}
  By Lemma \ref{irredstats}(2) the polynomial $G_n$ is monic, so
  $\omega$ is an algebraic integer.  Let $\p$ be the prime associated
  with $v$, and $\F_\p$ its residue field. Then the characteristic $p$
  of $\F_\p$ does not divide $2n-1$, so by Lemma \ref{irredstats}(5)
  the reduction of $g_n(\omega)$ to $\F_\p$ is not $0$. This implies
  $v(g_n(\omega))=0$.
\end{proof}

From Lemma \ref{vgntauzero} we find that if $\omega$ is a root of
$G_n$, and $v$ is some extension of the valuation associated to a
prime dividing $2n$, then the valuation at $v$ of the element in
(\ref{sqminoneeq}) equals $v(2n) + v(\omega-2)$.  The proof of the
following proposition shows that for odd $k$, in order to show that
$D(k,2n)$ is smooth, it suffices to note that this valuation is at
least $1$.

\begin{proposition}\label{proptwokponetwon}
  Let $k,l$ be any nonzero integers with $k$ odd, $l$ even, and
  $|k|\geq 2$.  Then the curve $D(k,l)$ is smooth over $\Q$.
\end{proposition}
\begin{proof}
  Set $m=(k-1)/2$ and $n=l/2$, so that $k=2m+1$ and $l=2n$.  Assume
  $P=(r_0,t_0)$ is a singular point over $\Qbar$ of the standard
  affine part of $D(k,l)$. Let $K$ be the number field $\Q(r_0,t_0)$,
  and let $v$ be the valuation on $K$ associated to a prime above $2$,
  normalized so that $v(2)=1$.  By Lemma \ref{firststep} we have
  $f_m(r_0) \neq 0\neq g_n(t_0)$ and $F_m(r_0)=G_n(t_0)=0$.  From
  Lemma \ref{vgntauzero} we then conclude $v(g_n(t_0))=0$.  Now around
  $P$ the curve $D(k,l)$ is given by $f_{m+1}(r)/f_m(r) =
  g_{n+1}(t)/g_n(t)$, which by (\ref{sqminoneeq}) and (\ref{Feq}) of
  Lemma \ref{lemmirrednewids} implies
\begin{align*}
  \frac{2m+1}{f_m(r_0)^2} &=\frac{2m+1+(r_0^2-4)F_m(r_0)}{f_m(r_0)^2}= \frac{f_{m+1}(r_0)^2-f_m(r_0)^2}{f_m(r_0)^2} \\
  &= \left(\frac{f_{m+1}(r_0)}{f_m(r_0)}\right)^2-1 =
  \left(\frac{g_{n+1}(t_0)}{g_n(t_0)}\right)^2-1 =
  \frac{2n(2-t_0)}{g_n(t_0)^2}.
\end{align*}
This contradicts the fact that the valuation at $v$ of the left-hand
side is at most $0$, while the valuation of the right-hand side is at
least $1$.  We conclude that no singular point $P$ exists on the
affine part.  By Lemma \ref{inftysmooth} there are also no singular
points at infinity.
\end{proof}

The only remaining case is the case that $k$ is even and $kl$ is
positive.  We deal with this case by investigating the possible values
of the valuation of the expression in (\ref{sqminoneeq}) at some
valuation extending $v_p$ for some prime $p$ dividing $2n$.

\begin{lemma}\label{binom}
  Let $n$ be a positive integer and $p$ a prime dividing $n$ and set
  $e=v_p(n)$.  Then for any integer $j\geq 0$ we have $v_p\big(\binom{n}{p^j}\big) = \max(e-j,0)$ and for any $0<k<p^j$ we have
  $v_p\left( \binom{n}{k} \right) > e-j$.
\end{lemma}
\begin{proof}
  For $j>e$ the statement is trivial, as $\binom{n}{k}$ is an
  integer, so we may assume $j\leq e$.  Let $l$ be any integer
  satisfying $1 \leq l \leq p^e$, and write $\binom{n}{l}$ as
$$
\binom{n}{l} = \frac{n}{l} \cdot \prod_{i=1}^{l-1} \frac{n-i}{i}.
$$
For all $i$ with $1\leq i < p^e$ we have $v_p(i) < v_p(n)$, so
$v_p(n-i) = v_p(i)$ and $v_p\big(\frac{n-i}{i}\big) = 0$. Therefore,
we have $v_p\left(\binom{n}{l}\right) = v_p(n) - v_p(l)$. Applying
this to $l=k$ and $l=p^j$, we obtain the statement, as $v_p(k) < j =
v_p(p^j)$.
\end{proof}

\begin{lemma}\label{newtonpoly}
  Let $n$ be a positive integer and $p$ a prime dividing $2n$.  Let
  $K$ be a number field and $v$ a valuation on $K$ with $v(p)=1$. Let
  $\alpha \in K$ satisfy $v(\alpha) = 0$ and set $e=v(4n)$.  If $p\neq
  2$, then also assume $v(2\alpha^{2n+1}+\alpha^2+1)=0$.  Then the
  Newton polygon of
$$
(S+\alpha)^{4n}+2n\left((S+\alpha)^{2n+1}-(S+\alpha)^{2n-1}\right) -1 = 
  \sum_{i=0}^{4n} b_i S^i
$$
at $v$ is the lower convex hull of the points
$$
\left\{
\begin{array}{ll}
  (0,v(b_0)),(1,e),(p,e-1),\ldots,(p^j,e-j),
\ldots,(p^e,0),(4n,0)& \mbox{if } p\neq 2,\\
  (0,v(b_0)),(1,v(b_1)),(2,v(b_2)),(3,e-1),(4,e-2),
\ldots,(2^e,0),(4n,0)& \mbox{if } p=2. 
\end{array}
\right.
$$
\end{lemma}
\begin{proof}
  The Newton polygon is the lower convex hull of all the points
  $(i,v(b_i))$ for $0\leq i \leq 4n$. It suffices to show that for
  each point $(i,a)$ in the given sequences we have $a = v(b_i)$,
  while for each $k$ for which there is no point $(k,a)$ in the
  sequence, there is a pair $(i_1,a_1)$, $(i_2,a_2)$ of consecutive
  points with $i_1 < k < i_2$, such that $v(b_k)\geq \max(a_1,a_2)$;
  this would certainly imply that the point $(k,v(b_k))$ is not below
  the line segment through $(i_1,a_1)$ and $(i_2,a_2)$. Note that for
  $k\geq 1$ we have
\begin{equation}\label{bk}
  b_k = \binom{4n}{k}\alpha^{4n-k} +
2n\left( \binom{2n+1}{k}\alpha^{2n+1-k}
    - \binom{2n-1}{k}\alpha^{2n-1-k}\right).
\end{equation}
Suppose $p\neq 2$, and let $(i,a)$ be a point in the corresponding
given sequence.  If $i=0$, then $a=v(b_0)$ by definition. We have $b_1
= 2n\alpha^{2n-2}(2\alpha^{2n+1}+\alpha^2+1+2n(\alpha^2-1))$. By
hypothesis we have $v(2\alpha^{2n+1}+\alpha^2+1)=0$ and as
$v(2n(\alpha^2-1))$ is positive, the valuation of the last factor of
$b_1$ is zero. Therefore, if $i=1$, then $v(b_i) = v(2n) =v(4n)= e=a$, as
needed.  If $i= p^j$ for $1\leq j \leq e$, then by Lemma \ref{binom}
the valuation of the first term in (\ref{bk}) for $k=i$ equals $e-j$,
while the valuation of the second term is at least $v(2n) = e$, so we
get $v(b_i) = e-j = a$, as needed. If $i=4n$, then $b_i=1$, so
$v(b_i)=0=a$, as needed. Suppose $k\leq 4n$ is an integer for which
there is no point $(k,a)$ in the sequence. If $k>p^e$, then all we
need to note is that $v(b_k)\geq 0$. If $k\leq p^e$, then there is a
$j \in \{1,2,\ldots, e\}$ such that $p^{j-1}<k<p^j$, in which case the
first term of (\ref{bk}) has valuation at least $e-j+1$ by Lemma
\ref{binom}, while the second term has valuation at least $e$, so we
have $v(b_k)\geq e-j+1 = \max(e-j+1,e-j)$, which is exactly what we
wanted to show.

Now suppose $p=2$, and let $(i,a)$ be a point in the corresponding
given sequence.  If $0\leq i\leq 2$, then $a=v(b_i)$ by
definition. Note that
$$
b_3 = \frac{2n\alpha^{2n-4}}{3}\left(2(8n^2-6n+1)\alpha^{2n+1}
  +n(4n^2-1)\alpha^{2} - (n-1)(4n^2-8n+3)\right).
$$
The first term between the parentheses has valuation $1$, while of the
second and third term, exactly one has valuation $1$, and the other has
valuation $0$, as exactly one of $n$ and $n-1$ is even. We conclude
that the expression between the parentheses has valuation $0$, so if
$i=3$, then $v(b_i) = v(2n) = e-1 = a$, as needed.  If $i= 2^j$ for
$2\leq j \leq e$, then by Lemma \ref{binom} the valuation of the first
term in (\ref{bk}) for $k=i$ equals $e-j$, while the valuation of the
second term is at least $v(2n) = e-1>e-j$, so we get $v(b_i) = e-j =
a$, as needed. If $i=4n$, then $b_i=1$, so $v(b_i)=0=a$, as
needed. Suppose $k\leq 4n$ is an integer for which there is no point
$(k,a)$ in the sequence. If $k>2^e$, then all we need to note is that
$v(b_k)\geq 0$. If $k\leq 2^e$, then there is a $j \in \{3,4,\ldots,
e\}$ such that $2^{j-1}<k<2^j$, in which case the first term of
(\ref{bk}) has valuation at least $e-j+1$ by Lemma \ref{binom}, while
the second term has valuation at least $e-1\geq e-j+1$, so we have
$v(b_k)\geq e-j+1 = \max(e-j+1,e-j)$, which is exactly what we wanted
to show.  This finishes the proof.
\end{proof}

\begin{lemma}\label{newtonone}
  Let $n$ be any positive integer and $p$ a prime. Set $e=v_p(4n)$.
  Then the Newton polygon of
$$
(S+1)^{4n}+2n\left((S+1)^{2n+1}-(S+1)^{2n-1}\right) -1 
$$
at $v_p$ has vertices 
$$
\left\{
\begin{array}{ll}
(0,\infty),(1,e),\ldots,(p^j,e-j),\ldots,(p^e,0),(4n,0)& 
\mbox{if } p\neq 2,\\
(0,\infty),(1,e+1),(4,e-2),\ldots,(2^j,e-j),\ldots,(2^e,0),(4n,0)& 
\mbox{if } p=2. 
\end{array}
\right.
$$
\end{lemma}
\begin{proof}
  The terms of lowest degree in the polynomial are $0S^0
  +8nS^1+4n(4n-1)S^2$.  If $p$ divides $2n$, then we can apply Lemma
  \ref{newtonpoly} with $\alpha=1$, and the result follows immediately
  from $b_0=0$, $b_1=8n$, and $b_2=4n(4n-1)$. If $p$ does not divide
  $2n$, then $p \neq 2$, and $e=0$, and $v_p(b_1) = 0$. It follows
  that the Newton polygon has vertices $(0,\infty)$, $(1,0)$, and
  $(4n,0)$, exactly as claimed.
\end{proof}

\begin{proposition}\label{notthreeok}
  Let $n$ be any positive integer and $p$ a prime. Let $K$ be a number
  field containing a root $\omega$ of $G_n$. Let $v$ be a valuation on
  $K$ with $v(p)=1$. Then we have $0\leq v(\omega-2) \leq 1$ and
  $v(h_{2n}(\omega)^2-1)\leq v(2n)+1$.  If $p$ divides $2n$, then we
  also have $v(2n)\leq v(h_{2n}(\omega)^2-1)$.  Moreover, if $p\neq 3$
  or $v(n)=0$, then the upper bounds in the first two inequalities are
  strict.
\end{proposition}
\begin{proof}
  By Lemma \ref{irredstats}(2) the root $\omega$ is an algebraic
  integer, so we have $v(2-\omega)\geq 0$ and $v(g_n(\omega))\geq 0$.
  From (\ref{sqminoneeq}) we know $h_{2n}(\omega)^2-1 =
  2n(2-\omega)g_n(\omega)^{-2}$, so $v(\omega-2) \leq 1$ implies
  $v(h_{2n}(\omega)^2-1)\leq v(2n)+1$ and if the former inequality is
  strict, then so is the latter. Also, if $p$ divides $2n$, then by
  Lemma \ref{vgntauzero} we have $v(g_n(\omega))=0$, so
  $v(h_{2n}(\omega)^2-1) = v(2n)+v(2-\omega)\geq v(2n)$. Therefore
  it suffices to show that $v(\omega-2) \leq 1$, and that this
  inequality is strict in the claimed cases.  Let $L$ be a finite
  field extension of $K$ containing a root $\sigma$ of $s^2-\omega s +
  1 =0$, and extend $v$ to $L$.  Then $\omega = \sigma + \sigma^{-1}$,
  so $\sigma$ is a root of $f = s^{4n}+2n(s^{2n+1}-s^{2n-1})-1$ by
  Lemma \ref{irredstats}(3).  This implies that $\sigma -1$ is a root
  of the polynomial in Lemma \ref{newtonone}, which we will denote by
  $F$. First consider the case $p=2$.  The polynomial $f$ has roots
  $1$ and $-1$ of multiplicity $1$ and $3$ respectively, corresponding
  to roots $0$ and $-2$ of $F$, which in turn correspond to the line
  segments of the Newton polygon from $(0,\infty)$ to $(1,e+1)$ and
  from $(1,e+1)$ to $(4,e-2)$ respectively by Lemma \ref{polyroots}.
  If $\sigma$ were one of these roots of $f$, then we would have
  $\omega = \pm 2$, which contradicts $G_n(2)=n$ and $G_n(-2) =
  \frac{1}{3}n(4n^2-1)$ by Lemma \ref{irredstats}.  The root
  $\sigma-1$ of $F$ therefore corresponds to another segment of the
  Newton polygon of $F$, all of which have slope between
  $-\frac{1}{4}$ and $0$, so we have $0\leq v(\sigma-1) \leq
  \frac{1}{4}$ by Lemma \ref{polyroots}, and thus $v(\omega-2) =
  v(\sigma^{-1}(\sigma - 1)^2) = 2v(\sigma-1) \leq \frac{1}{2}<1$. Now
  consider the case $p>2$. We still have $\sigma -1 \neq 0$, so the
  root $\sigma -1$ of $F$ corresponds to a nonvertical segment of the
  Newton polygon of $F$. These segments all have slope equal to $0$ or
  $-1/(p^j-p^{j-1})$ for some $1\leq j \leq v(n)$, so we have
  $v(\sigma-1) \leq \frac{1}{p^j-p^{j-1}}$, and $0\leq v(\omega-2) =
  2v(\sigma-1) \leq \frac{2}{p^j-p^{j-1}}\leq 1$, where the equality
  follows as it did in case $p=2$. The last inequality is strict
  unless $p=3$ and $j=1$, in which case $v(n)>0$. This proves the
  proposition.
\end{proof}

If $k$ and $l$ are even and $kl$ is positive, and $k$ and $l$ are not
equal and do not differ by a factor $3$, then the results above are
sufficient to show that there exists a prime $p$ such that the values
at critical points of $h_k$ are different from those of $h_l$. This
would show that $D(k,l)$ is smooth over $\Q$. The following results
allow us to also handle the case that $k$ and $l$ differ by a factor
$3$.

\begin{lemma}\label{newtoni}
  Let $n$ be any positive integral multiple of $3$.  Let $K$ be the
  number field $\Q(i) = \Q[x]/(x^2+1)$ and let $v$ be the unique
  valuation on $K$ satisfying $v(3)=1$. Set $e=v(n)$.  Then the Newton
  polygon of
$$
(S+i)^{4n}+2n\left((S+i)^{2n+1}-(S+i)^{2n-1}\right) -1 
$$
at $v$ has vertices 
$$
(0,e),(3,e-1),\ldots,(3^j,e-j),\ldots,(3^e,0),(4n,0).\\
$$
\end{lemma}
\begin{proof}
This follows immediately from Lemma \ref{newtonpoly}
with $\alpha = i$ and $p=3$. 
\end{proof}

\begin{proposition}\label{minn}
  Let $n$ be any positive integral multiple of $3$.  Let $K$ be a
  number field containing a root $\omega$ of $G_n$. Let $v$ be a
  valuation on $K$ with $v(3)=1$. Then $v(2n) \leq
  v(h_{2n}(\omega)^2-1- n) < v(2n)+1$.
\end{proposition}
\begin{proof}
  From (\ref{sqminoneeq}) we deduce $h_{2n}(\omega)^2-1-n =
  ng_n(\omega)^{-2} A$ with $A = 4-2\omega-g_n(\omega)^2$. By Lemma
  \ref{vgntauzero} we have $v(g_n(\omega))=0$, so
  $v(h_{2n}(\omega)^2-1-n) = v(n) + v(A)$. As $\omega$ is an algebraic
  integer, we have $v(A)\geq 0$, so it suffices to show $v(A) <1$. Let
  $L$ be a finite field extension of $K$ containing a square root $i$
  of $-1$ and a root $\sigma$ of $s^2-\omega s + 1 =0$, and extend $v$
  to $L$.  Let $R$ denote the discrete valuation ring of $L$
  associated to $v$, and $\m$ its maximal ideal. For each $\varepsilon
  \in \{\pm 1\}$ we have $-(\sigma+1)^2A =
  X(\varepsilon)+Y(\varepsilon)+Z(\varepsilon)$ with
  $X(\varepsilon)=2\sigma^{-1}(\sigma-\varepsilon)^2(\sigma^2+1)$,
  $Y(\varepsilon)=3\varepsilon(\sigma-\varepsilon)^2$, and
  $Z(\varepsilon)=\varepsilon(\sigma^{2n}+\varepsilon)
  (\sigma^{2-2n}+\varepsilon)$.

  We have $\omega = \sigma + \sigma^{-1}$, so $\sigma$ is a root of $f
  = s^{4n}+2n(s^{2n+1}-s^{2n-1})-1$ by Lemma \ref{irredstats}(3).
  This implies that $\sigma-i$ is a root of the polynomial in Lemma
  \ref{newtoni}.  Since the slopes of the Newton polygon of this
  polynomial are between $-\frac{1}{3}$ and $0$ by Lemma
  \ref{newtoni}, we have $v(\sigma-i) \leq \frac{1}{3}$ by Lemma
  \ref{polyroots}.  Replacing $i$ by $-i$ temporarily, we also find
  $v(\sigma+i) \leq \frac{1}{3}$, so we get $v(\sigma^2+1) =
  v(\sigma+i)+v(\sigma-i) \leq \frac{2}{3} < 1$. From $f(\sigma)=0$ we
  get $(\sigma^{2n}-1)(\sigma^{2n}+1)=-2n\sigma^{2n-1}(\sigma^2-1)$.
  As the elements $\sigma^{2n}-1$ and $\sigma^{2n}+1$ differ by $2$,
  which is a unit in $R$, at least one of them is also a unit, with
  valuation $0$, so we conclude
\begin{align*}
  \max\big(&v(\sigma^{2n}+1),v(\sigma^{2n}-1)\big) =
  v(\sigma^{2n}+1)+v(\sigma^{2n}-1) \\
  &=v\big((\sigma^{2n}+1)(\sigma^{2n}-1)\big) 
  =v\big(-2n\sigma^{2n-1}(\sigma^2-1)\big) \\
  &=v(n) + v(\sigma^2-1) \geq 1 + v(\sigma^2-1) \geq 1+v(\sigma+1).
\end{align*}
%
%
Suppose first that $\sigma^{2n}-1$ is a unit, and thus that
$v(\sigma^{2n}+1) \geq 1$. Since $\sigma^{2n}-1$ is a multiple of
$\sigma^2-1$ in $R$, we find that $\sigma^2-1$ is also a unit, so
$v(\sigma+1)=v(\sigma-1)=0$.
%
%
We get $v(X(1)) = v(\sigma^2+1) < 1$, while $v(Y(1)), v(Z(1)) \geq 1$,
so we obtain $v(A) = v(-(\sigma+1)^2A) = v(X(1)+Y(1)+Z(1)) = v(X(1)) <
1$ and we are done. Hence we may assume that $\sigma^{2n}-1$ is not a
unit, so $\sigma^{2n}+1$ is a unit and we have $v(\sigma^{2n}-1) \geq
1 + v(\sigma+1)$. Since $\sigma^{2-2n}-1$ is a multiple of $\sigma+1$
we also have $v(\sigma^{2-2n}-1) \geq v(\sigma+1)$ and thus $v(Z(-1))
\geq (1+ v(\sigma+1)) +v(\sigma+1) \geq 1+2v(\sigma+1)$.  We also have
$v(Y(-1)) = 1+2v(\sigma+1)$ and $v(X(-1)) =v(\sigma^2+1)+2v(\sigma+1)<
1 + 2v(\sigma+1)$. This yields $v(A) = v(-(\sigma+1)^2A) -
2v(\sigma+1) = v(X(-1)+Y(-1)+Z(-1)) - 2v(\sigma+1) = v(X(-1)) -
2v(\sigma+1) < 1$, which finishes the proof.
\end{proof}

\begin{lemma}\label{newtonalpha}
  Let $n$ be any positive integral multiple of $3$.  Let $K$ be the
  number field $\Q[x]/(x^2-3)$, let $\beta$ be the image of $x$ in
  $K$, and let $v$ be the unique valuation on $K$ satisfying
  $v(3)=1$. Set $e=v(n)$ and $\alpha = -2+\beta$.  Then the Newton
  polygon of
$$
(S+\alpha)^{4n}+2n\left((S+\alpha)^{2n+1}-(S+\alpha)^{2n-1}\right) -1 
$$
at $v$ has vertices 
$$
(0,e+\textstyle{\frac{3}{2}}),(1,e),(3,e-1),\ldots,(3^j,e-j),
          \ldots,(3^e,0),(4n,0).\\
$$
\end{lemma}
\begin{proof}
  By Lemma \ref{newtonpoly} it suffices to check $v(a)=0$ with $a =
  2\alpha^{2n+1}+\alpha^2+1$, and $v(b_0) = e+\frac{3}{2}$ with $b_0 =
  \alpha^{4n} +2n(\alpha^{2n+1}-\alpha^{2n-1})-1$.  Note that we have
  $\alpha = 1+\gamma$ with $\gamma = \beta-\beta^2$, while $\beta$
  generates the ideal $\p$ to which $v$ is associated. It follows that
  $\alpha \equiv 1 \pmod{\p}$, so $a \equiv 1 \pmod{\p}$, and indeed
  $v(a)=0$. Expanding the powers of $\alpha = 1+\gamma$ gives $b_0 =
  \sum_{i=1}^{4n} c_i \gamma^i$ with $c_i = \binom{4n}{i} +
  2n\binom{2n+1}{i} -2n \binom{2n-1}{i}$. We claim that for
  $i\geq 4$ we have $v\big(\binom{4n}{i}\big) \geq e+2-i/2$.  Write
  $\binom{4n}{i} = 4n\cdot \frac{1}{i!}(4n-1)\cdots (4n-i+1)$ and
  note that the product of at least three consecutive integers is
  divisible by $3$.  As $v(4!) = v(5!) = 1$, we conclude that for
  $i=4,5$ we have $v\big(\binom{4n}{i}\big) \geq v(n) =e \geq
  e+2-i/2$.  Suppose $i\geq 6$ and let $j\geq 2$ be the integer
  satisfying $3^{j-1}\leq i <3^j$. Since $v\big(\binom{4n}{i}\big)$
  is an integer, we have $v\big(\binom{4n}{i}\big) \geq e+1-j \geq
  e+2-i/2$ by Lemma \ref{binom}, where the last inequality follows
  from $i\geq 6$ for $j=2$ and from $i\geq 3^{j-1} \geq 2j+2$ for
  $j\geq 3$.  This proves the claim, and as the last two terms of
  $c_i$ have valuation at least $v(2n)=e$, we find $v(c_i)\geq
  e+2-i/2$ for $i\geq 4$. This gives $v(c_i\gamma^i) \geq e+2-i/2 +
  i\cdot v(\gamma) = e+2$ for $i\geq 4$, and therefore
  $v(\sum_{i=4}^{4n} c_i \gamma^i) \geq e+2$. We also have
$$
c_1\gamma+c_2\gamma^2+c_3\gamma^3 = 4n
\big( (140\beta - 252)n^2 - (144\beta - 264)n + 33\beta - 63 \big),
$$
which has valuation $e+\frac{3}{2}$, as there is a unique term with
lowest valuation $\frac{3}{2}$ inside the parentheses, namely
$33\beta$.  We conclude $v(b_0) = e + \frac{3}{2}$.
\end{proof}

\begin{proposition}\label{minthreen}
  Let $n$ be any positive integral multiple of $3$.  Let $K$ be a
  number field containing a root $\omega$ of $G_n$.  Let $v$ be a
  valuation on $K$ with $v(3)=1$. Then $v(n) \leq
  v(h_{2n}(\omega)^2-1-3n) < v(n)+1$ or $v(h_{2n}(\omega)^2-1-3n) 
  \geq v(n)+2$.
\end{proposition}
\begin{proof}
  From (\ref{sqminoneeq}) we deduce $h_{2n}(\omega)^2-1-3n =
  ng_n(\omega)^{-2} A$ with $A = -2(\omega+4)+3(4-g_n(\omega)^2)$. By
  Lemma \ref{vgntauzero} we have $v(g_n(\omega))=0$, so
  $v(h_{2n}(\omega)^2-1-3n) = v(n) + v(A)$. As we clearly have
  $v(A)\geq 0$, it suffices to show $v(A) <1$ or $v(A) \geq 2$. Let
  $L$ be a finite field extension of $K$ containing a square root
  $\beta$ of $3$ and a root $\sigma$ of $s^2-\omega s + 1 =0$. Set
  $\alpha = -2+\beta$ and $\overline{\alpha} = -2-\beta =
  \alpha^{-1}$, and extend $v$ to $L$.  Let $F$ denote the polynomial
  in Lemma \ref{newtonalpha}, and set $f(s) = s^{4n} +
  2n(s^{2n+1}-s^{2n-1}) - 1$, so that $F(S) = f(S+\alpha)$.  From
  Lemma \ref{polyroots} and the slopes of the Newton polygon of $F$
  given in Lemma \ref{newtonalpha}, we conclude that there is a unique
  root $S_0$ of $F$ with $v(S_0) = \frac{3}{2}$.  Then $s_0 =
  S_0+\alpha$ is a root of $f$, and as $f$ is anti-reciprocal, so is
  $s_1 = s_0^{-1}$ and both are units in the ring of integers of
  $L$. Set $S_1 = s_1-\alpha$.  Then $S_1$ is root of $F$ and from the
  identity $S_1= -\overline{\alpha}s_0^{-1}S_0 - 2\beta$ and the
  inequality $v(2\beta ) = \frac{1}{2} <
  v(\overline{\alpha}s_0^{-1}S_0)$ we conclude $v(S_1) =
  \frac{1}{2}$. Note that $S_2=3-\beta$ is also a root of $F$,
  corresponding to the root $1$ of $f$, and with
  $v(S_2)=\frac{1}{2}$. By Lemma \ref{polyroots} and Lemma
  \ref{newtonalpha} there are only three roots $z$ of $F$ with
  $v(z)\geq \frac{1}{2}$, so all roots $z$ of $F$, other than
  $S_0,S_1,S_2$, satisfy $v(z)<\frac{1}{2}$.

  Now $\omega = \sigma+\sigma^{-1}$, so by Lemma \ref{irredstats}(3),
  the element $\sigma$ is a root of $f$, and therefore $\sigma
  -\alpha$ is a root of $F$.  First suppose the inequality $v(\sigma
  -\alpha) < \frac{1}{2}=v(2\beta)$ holds.  Then we also have
  $v(\sigma -\overline{\alpha}) = v(\sigma - \alpha +2\beta) <
  \frac{1}{2}$, and thus $v(\omega+4) =
  v\big(\sigma^{-1}(\sigma-\alpha)(\sigma-\overline{\alpha})\big) <
  1$.  From $0\leq v(2(\omega+4))<1\leq v(3(4-g_n(\omega)^2))$ we
  conclude $v(A) = v(\omega+4) <1$ and we are done.

  Now suppose $v(\sigma -\alpha) \geq \frac{1}{2}$, then $\sigma
  -\alpha = S_i$ for some $i$ with $0\leq i \leq 2$. For $i=2$ we get
  $\sigma = S_2+\alpha = 1$ and thus $\omega=2$, so $G_n(\omega)= n
  \neq 0$ by Lemma \ref{irredstats}(4). From this contradiction we
  conclude $\sigma = s_i$ for $i=0$ or $i=1$, so that $\omega+4 =
  \sigma+\sigma^{-1}+4 = s_0+s_1+4 = -\overline{\alpha}S_0S_1$. This
  implies $v(\omega+4)=\frac{3}{2}+\frac{1}{2} = 2$.  We rewrite $A$
  as
\begin{equation}\label{reA}
  A=-2(\omega+4) + 9 - 3(\omega-2)
\left(\frac{g_n(\omega)^2-1}{\omega-2}\right).
\end{equation}
From Lemma \ref{fgat2} we know that $d(t) = (g_n(t)-1)/(t-2)$ is a
polynomial, so $v((g_n(\omega)^2-1)/(\omega-2)) =
v((g_n(\omega)+1)d(\omega)) \geq 0$. From $\omega-2= (\omega+4)-6$ we
get $v(\omega-2)=v(6)=1$, so the last term of (\ref{reA}) has
valuation at least $2$, while $v(-2(\omega+4)) = v(9) = 2$. We
conclude $v(A) \geq 2$, which finishes the proof.
\end{proof}

\begin{proposition}\label{proptwoktwon}
  Let $k,l$ be any even integers with $k\neq l$ and $kl>0$. Then
  $D(k,l)$ is smooth over $\Q$.
\end{proposition}
\begin{proof}
  The curve $D(k,l)$ is the same as $D(-k,-l)$, so without loss of
  generality we assume $k,l>0$.  Set $m=k/2$ and $n=l/2$ and
  $F=g_{m+1}(r)g_n(t)-g_m(r)g_{n+1}(t)$.  Assume $P=(r_0,t_0)$ is a
  singular point over $\Qbar$ of the standard affine part of
  $D(k,l)$. Let $K$ be the number field $\Q(r_0,t_0)$.  By Lemma
  \ref{firststep} we have $G_n(t_0)=0$ and $G_m(r_0)=0$, and $D(k,l)$
  is given around $P$ by $h_k(r) = h_l(t)$.  Set $c = h_k(r_0)^2-1$
  and $d=h_l(t_0)^2-1$.  Let $p$ be any prime such that $v_p(m) \neq
  v_p(n)$. Set $e=v_p(n)-v_p(m)$.  By symmetry we may assume $e\geq
  1$. Let $\p$ be a prime of $K$ above $p$, and let $v$ be the
  valuation on $K$ associated to $\p$, normalized so that $v$
  restricts to $v_p$ on $\Q$.  By Lemma \ref{notthreeok} we have $v(c)
  \leq v(2m)+1 \leq v(2m)+e = v(2n) \leq v(d)$. From $c=d$ we conclude
  that all inequalities are equalities, so $e=1$ and by Lemma
  \ref{notthreeok} we have $p|m$ and $p=3$, and thus
  $n=3m$. Proposition \ref{minn} shows $v(2m)+1=v(2n)\leq v(d-n) <
  v(2n)+1 = v(2m)+2$, while from Proposition \ref{minthreen} we get
  $v(c-3m) < v(2m)+1$ or $v(c-3m) \geq v(2m)+2$. This contradicts the
  equality $c-3m=d-n$, and we conclude that no singular point $P$
  exists on the affine part.  By Lemma \ref{inftysmooth} there are
  also no singular points at infinity.
\end{proof}

We have now proved the first statement of Theorem \ref{thmDsmooth},
split over several Propositions. To prove the last statement, we set
set $H_n = g_{n+1}''g_n-g_{n+1}g_n''$ for each integer $n$, where
derivatives are taken with respect to $u$.  In $\Z[u][s]/(s^2-us+1)
\isom \Z[s,s^{-1}]$ one checks
\begin{equation}
\frac{1}{2}(u-2)(u+2)^2H_n=(n-1)f_{2n+1}+f_{2n}-(n+1)f_{2n-1}-nu+2n.
\label{Heq}
\end{equation}
Recall that for even $l\neq 0$, the curve $D_1(l,l)$ is the projective
closure of the scheme-theoretic complement in $D(l,l)$ of the line
given by $r=t$.

\begin{proposition}\label{Donesmooth}
Let $l$ be any even integer with $|l|\geq 4$. Then the curve 
$D_1(l,l)$ is smooth over $\Q$. 
\end{proposition}
\begin{proof}
  Set $n=l/2$ and $F=g_{n+1}(r)g_n(t)-g_n(r)g_{n+1}(t)$ and
  $G=F/(t-r)$.  Then $D_1(l,l)$ is defined by $G(r,t)=0$. Any singular
  point of $D_1(l,l)$ is also a singular point of $D(l,l)$.  By Lemma
  \ref{inftysmooth} we find that $D(l,l)$ is smooth at all points at
  infinity, so $D_1(l,l)$ is as well.  Assume $P=(r_0,t_0)$ is a
  singular point of the standard affine part of $D_1(l,l)$. Then $P$
  is also a singular point of $D(l,l)$.  By Lemma \ref{firststep} we
  then have $G_n(t_0)=0$ and $G_n(r_0)=0$, and we may rewrite $F(P)=0$
  as $h_l(r_0) = h_l(t_0)$.  Then from (\ref{surprise}) of Lemma
  \ref{lemmirrednewids} we have
$$
2nr_0 = (2n-1)h_l(r_0)+(2n+1)h_l(r_0)^{-1} = 
(2n-1)h_l(t_0)+(2n+1)h_l(t_0)^{-1} = 2nt_0,
$$
which implies $r_0 = t_0$. Set $F_t = \partial F/\partial t$ and
$F_{t^2} = \partial F_t/\partial t$ and $G_t = \partial G/\partial t$.
Then we have $G(r_0,r_0) = F_t(r_0,r_0) = G_n(r_0)$, where the first
equality can be viewed as an algebraic version of l'H\^opital's rule
applied to $\lim_{t\rightarrow r_0} G(r_0,t)$.  That same rule also
gives $G_t(r_0,r_0) = \frac{1}{2}F_{t^2}(r_0,r_0) =
\frac{1}{2}H_n(r_0)$. The fact that $D_1(l,l)$ is singular at $P$
implies $0=G(P)=G_n(r_0)$ and $0=G_t(P) = \frac{1}{2}H_n(r_0)$. From
Lemma \ref{irredstats}(3) and (\ref{Heq}) we then deduce
\begin{align*}
0 &= (r_0+2)\left(\left((n-1)r_0+1\right)G_n(r_0)-
\frac{1}{2}(r_0^2-4)H_n(r_0)\right) \\
  &=  n \left( 2f_{2n-1}(r_0) + (2n-1)r_0 \right),
\end{align*}
so we get $r_0 = 2(f_{2n-1}(r_0) + n)$. This implies $v(r_0)\geq 1$
for any valuation $v$ of $\Q(r_0)$ with $v(2)=1$, which contradicts
the inequality $v(r_0-2) < 1$ from Proposition \ref{notthreeok}.  We
conclude that $D_1(l,l)$ has no singular points.
\end{proof}

\begin{proof}[Proof of Theorem \ref{thmDsmooth}]
  The first statement follows immediately from Propositions
  \ref{oppsigns}, \ref{proptwokponetwon}, \ref{proptwoktwon}, while
  the last statement follows from Proposition \ref{Donesmooth}.
\end{proof}

\section{Genera of the irreducible components}
\label{generairred}

The following theorem tells us the number of irreducible components of
$Y(k,l)$ in all cases that $J(k,l)$ is a hyperbolic knot. Recall that
$Y_0(l,l)$ and $Y_1(l,l)$ were defined in Definition \ref{zeros}.

\begin{theorem}\label{thmirredcomps}
  Let $k,l$ be any nonzero integers with $l$ even, $|k|\geq 2$, and
  $k\neq l$.
\begin{enumerate}
\item The curve $D(k,l)$ is a smooth, projective, geometrically 
irreducible curve of bidegree $(\lfloor |k|/2\rfloor , |l|/2)$ 
containing an open subset that is isomorphic to $Y(k,l)=Y_0(k,l)$.
\item If $|l|>2$, then $D_1(l,l)$ is a smooth, projective,
  geometrically irreducible curve of bidegree $(|l|/2-1,|l|/2-1)$
  containing an open subset that is isomorphic to $Y_1(l,l)$. The
  curve $Y(l,l)$ consists of two geometrically irreducible components,
  namely $Y_0(l,l)$ and $Y_1(l,l)$.
\end{enumerate}
\end{theorem}
%
\begin{proof}
  The curves $D(k,l)$ and $D(l,l)$ are projective by construction.  By
  Theorem \ref{thmDsmooth} the curve $D(k,l)$ is smooth and its
  bidegree is given in Remark \ref{bidegrees}. Every smooth projective
  curve in $\P^1 \times \P^1$ of bidegree $(a,b)$ with $a,b>0$ is
  geometrically irreducible by Lemma \ref{genushyp}, so $D(k,l)$ is
  geometrically irreducible.  By Lemma \ref{Csmooth} and Proposition
  \ref{charvar} the curve $C(k,l)$ is isomorphic to an open subset of
  $D(k,l)$.  Since $Y(k,l)$ is isomorphic to $C(k,l)$, we conclude
  that $Y(k,l)$ is isomorphic to an open subset of $D(k,l)$, so
  $Y(k,l)$ is geometrically irreducible and smooth as well, and
  therefore equal to $Y_0(k,l)$.  Suppose $|l|>2$. By Theorem
  \ref{thmDsmooth} the curve $D_1(l,l)$ is smooth and its bidegree is
  given at the end of \S \ref{newmodelsection}.  By Lemma
  \ref{genushyp} the curve $D_1(l,l)$ is geometrically irreducible, so
  $D(l,l)$ consists of two irreducible components, namely $D_0(l,l)$
  and $D_1(l,l)$, cf. end of \S \ref{newmodelsection}.
  By Proposition \ref{charvar}, the curve $C(l,l)$ is birationally
  equivalent to $D(l,l)$, so it also has two components, one of which
  is isomorphic to a subset of $D_1(l,l)$ by Lemma \ref{Csmooth}.
  Since $Y(l,l)$ is isomorphic to $C(l,l)$, the curve $Y(l,l)$ also has 
  two components, so $Y_1(l,l)$ is irreducible and the components are
  $Y_0(l,l)$ and $Y_1(l,l)$.  Since $Y_0(l,l)$ corresponds to
  $D_0(l,l)$ by Proposition \ref{Yzero}, it is $Y_1(l,l)$ that is
  isomorphic to a subset of $D_1(l,l)$.
\end{proof}

It is now easy to find the genus of the components of $Y(k,l)$.

\begin{theorem}\label{genY}
Let $k,l$ be any nonzero integers with $l$ even, $|k|\geq 2$, and
$k\neq l$.
\begin{enumerate}
\item The curve $Y(k,l)=Y_0(k,l)$ has geometric genus $(\lfloor
  |k|/2\rfloor-1)(|l|/2-1)$ and is hyperelliptic if and only if
  $|k|\leq 5$ or $|l|\leq 5$.
\item If $|l|>2$, then the curve $Y_0(l,l)$ has genus $0$ and the
  curve $Y_1(l,l)$ has genus $(|l|/2-2)^2$ and is hyperelliptic if and
  only if $|l|\leq 6$.
\end{enumerate}
\end{theorem}
\begin{proof}
  By Theorem \ref{thmirredcomps} the curves $Y(k,l)$ and $Y_0(k,l)$
  are both birationally equivalent to $D(k,l)$, which is a smooth
  irreducible curve in $\P^1 \times \P^1$ of bidegree $(\lfloor |k|/2
  \rfloor,|l|/2)$. Statement (1) therefore follows from Lemma
  \ref{genushyp}.  By Proposition \ref{Yzero} the curve $Y_0(l,l)$ is
  birationally equivalent to a line, so it has genus $0$. By Theorem
  \ref{thmirredcomps} the curve $Y_1(l,l)$ is birationally equivalent
  to $D_1(l,l)$, which is smooth of bidegree $(|l|/2-1, |l|/2-1)$.
  Therefore, statement (2) follows from Lemma \ref{genushyp}.
\end{proof}

\begin{proof}[Proof of Theorem \ref{mainone}]
This follows immediately from Theorem \ref{genY}.
\end{proof}

%
%
%

Our next goal is to investigate the ramification of the map from
$X(k,l)$ to $Y(k,l)$, which we will then use to compute the genus of
the irreducible components of $X(k,l)$. The component $X_0(k,l)$ lies
above $Y_0(k,l)$. For $|l|> 2$ we know that $Y(l,l)$ consists of two
irreducible components, so $X(l,l)$ consists of at least two
components.

%
%

\begin{lemma}\label{heven}
  Let $m,n$ be any nonzero integers. Consider the function $h =
  (r-2)\big(2-t + (r^2-4)f_m(r)^2 \big)$ on $D(2m,2n)$.  Then $h$ is
  regular and nonvanishing at all singular points of $D(2m,2n)$, and
  has odd valuation at exactly $2|mn|+2|m|+2|n|-2a$ nonsingular points
  of $D(2m,2n)$, with $a=2$ for $mn>0$ and $a=1$ for $mn<0$.  If
  $m=n$, then exactly $2|n|$ of these points lie on the line
  determined by $r=t$.
\end{lemma}
\begin{proof}
  Set $d = 2-t + (r^2-4)f_k(r)^2 $, so that $h = (r-2)d$.  Let $M$ and
  $C$ denote the vanishing locus of $r-2$ and $d$ respectively.  From
  $g_m(2) = g_{m+1}(2) = 1$, we see that there are $|n|$ points in the
  affine part of the intersection $M\cap D(2m,2n)$, namely $(2,\tau)$
  for each root $\tau$ of $g_{n+1}(t) - g_n(t)$. As $D(2m,2n)$ has
  bidegree $(|m|,|n|)$, we have $M\cdot D(2m,2n) = |n|$, which shows
  that all intersection multiplicities are trivial, so all
  intersection points are smooth, and the intersections are
  transversal. This implies that for each $Q$ of these $|n|$ points we
  have $v_Q(r-2)=1$ by Lemma \ref{tranint}.  The points in the
  standard affine part of the intersection $C\cap D(2m,2n)$ correspond
  to the roots of $F(r,T)$ with $F(r,t)=
  g_m(r)g_{n+1}(t)-g_{m+1}(r)g_n(t)$ and $T = (r^2-4)f_m(r)^2 +
  2$. The degree of $T$ equals $2|m|$.  From Lemma \ref{newbasicf} we
  find that the degree of $F(r,T)$ as a polynomial in $r$ equals
  $2|mn|+|m|+1-a$. We now show that $F(r,T)$ is separable. Consider
  the extension $\Z[r][q]/(q^2-rq+1)\isom \Z[q,q^{-1}]$. Then we have
  $r=q+q^{-1}$ and from $f_m(r) = (q^m-q^{-m})/(q-q^{-1})$ we find $T
  = q^{2m}+q^{-2m}$. This yields $F(r,T) =
  q^{1-m-2mn}(q^{2m}-1)(q^{4mn-1}-1)/(q+1)$, and as $\gcd(2m,4mn-1)
  =1$, we find the only multiple factor of $F(r,T)$ in $\Z[q,q^{-1}]$
  is $(q-1)^2 = q(r-2)$, which corresponds to the single root
  $r=2$. We conclude that $F(r,T)$ is indeed separable. This shows
  that all $2|mn|+|m|+1-a$ intersection points $R$ in the affine part
  $C\cap D(2m,2n)$ are transversal intersections, so they are smooth
  points of $D(2m,2n)$, and we have $v_R(d)=1$ by Lemma
  \ref{tranint}. If $h$ vanishes, then either $r-2$ or $d$ does.  The
  only point where both $r-2$ and $d$ vanish is $P=(2,2)$, where the
  valuation $v_P(h) = v_P(r-2)+v_P(d) = 2$ is even.  At the remaining
  $(|n|-1) +(2|mn|+|m|-a) = 2|mn|+|m|+|n|-1-a$ points $S$ where $r-2$
  or $d$ vanishes, the valuation $v_S(h) = v_S(r-2) + v_S(d) = 1$ is
  odd.

  Let $L_r$ and $L_t$ denote the lines given by $r=\infty$ and
  $t=\infty$ respectively. Lemma \ref{inftysmooth} tells us that $L_r$
  and $L_t$ intersect $D(2m,2n)$ transversally everywhere, so $1/r$ is
  a uniformizer at every point in $L_r \cap D(2m,2n)$, while $1/t$ is
  a uniformizer at every point in $L_t \cap D(2m,2n)$. This shows
  $v_S(r^it^j)=-i$ for every point $S$ in $L_r\cap D(2m,2n)$ that is
  not on $L_t$, while $v_S(r^it^j)=-j$ for every point $S$ in $L_t\cap
  D(2m,2n)$ that is not on $L_r$, and $v_S(r^it^j)=-i-j$ for the
  unique point in $L_r \cap L_t \cap D(2m,2n)$, if it exists. We
  obtain $v_S(h) = -2|m|-1$ and $v_S(h) = -1$ and $v_S(h) = -2|m|-1$
  for these three cases respectively. There are $|n|$ points in $L_r
  \cap D(2m,2n)$ and $|m|$ points in $L_t \cap D(2m,2n)$, while the
  overlap $L_r \cap L_t \cap D(2m,2n)$ contains a point if and only if
  $mn>0$. This gives a total of $|m|+|n|+1-a$ points $S$ at infinity,
  all with $v_S(h)$ odd. Together with the affine points this makes
  $2|mn|+2|m|+2|n|-2a$ points where $h$ has odd valuation. Suppose
  $m=n$.  Then $2|n|$ of these points lie on the line given by $r=t$,
  namely the point in $L_r \cap L_t$, and the $2|n|-1$ points
  $(r_0,r_0)$ for all roots $r_0\neq 2$ of $T-r$.
\end{proof}

\begin{lemma}\label{hodd}
  Let $m,n$ be any nonzero integers with $m\not \in \{-1,0\}$.
  Consider the function $h = t -2+ (r+2)g_{m+1}(r)^2$ on
  $D(2m+1,2n)$. Then $h$ has odd valuation at exactly
  $|2m+1|\cdot|n|+|n|+2|m|-2a$ points of $D(2m+1,2n)$, with $a=1$ when
  $n>0$ and $a=2$ when $m,n<0$ and $a=0$ when $n<0<m$.
\end{lemma}
\begin{proof}
  Let $C$ denote the vanishing locus of $h$.  The points in the
  standard affine part of the intersection $C\cap D(2m,2n)$ correspond
  to the roots of $F(r,T)$ with $F(r,t)=
  f_m(r)g_{n+1}(t)-f_{m+1}(r)g_n(t)$ and $T =
  2-(r+2)g_{m+1}(r)^2$. The degree of $T$ equals $|2m+1|$.  From Lemma
  \ref{newbasicf} we find that the degree of $F(r,T)$ equals
  $|2m+1|\cdot|n|+|m|-a$.  We now show that $F(r,T)$ is separable.
  Consider the extension $\Z[r][q]/(q^2-rq+1)\isom \Z[q,q^{-1}]$. Then
  we have $r=q+q^{-1}$ and $T = -q^{2m+1}-q^{-2m-1}$. This yields
  $F(r,T) = -q^{2mn+m+n-1}(q^{2m+1}+1)(q^{2(2m+1)n-1}-1)/(q^2-1)$, and
  as $\gcd(2(2m+1),2(2m+1)n-1)=1$, we find that $F(r,T)$ has no
  multiple factors in $\Z[q,q^{-1}]$, so $F(r,T)$ is indeed separable.
  This shows that all $|2m+1|\cdot|n|+|m|-a$ intersection points $R$
  in the affine part $C\cap D(2m+1,2n)$ are transversal intersections,
  and we have $v_R(h)=1$ by Lemma \ref{tranint}.

  Let $L_r$ and $L_t$ denote the lines given by $r=\infty$ and
  $t=\infty$ respectively. As in the proof of Lemma \ref{heven}, the
  valuation $v_S(h)$ is odd at every point $S$ at infinity.  There are
  $|m|+|n|-a$ points at infinity, so we get a total of
  $|2m+1|\cdot|n|+|n|+2|m|-2a$ points $S$ with $v_S(h)$ odd.
\end{proof}

We now have enough information to compute the genus of the irreducible
components of $X(k,l)$ for any $k,l$ for which $J(k,l)$ is a
hyperbolic knot.  Recall that if $l$ is an even integer with$|l|>2$,
then $X_1(l,l)$ is the scheme-theoretic complement of $X_0(l,l)$ in
$X(l,l)$.

\begin{theorem}\label{genXdiff}
  Suppose $l$ is a nonzero even integer, say $l=2n$.  If $k\neq l$ is
  an integer satisfying $|k|\geq 2$, then $X(k,l)$ is irreducible and
  the genus of $X_0(k,l)$, its only irreducible component, equals
$$
3|mn|-|m|-a|n|+b,
$$
with $m = \lfloor k/2 \rfloor$ and
$$
a = \left\{
\begin{array}{rl}
4 & \mbox{if $k$ is odd and $k<0$},\cr
1 & \mbox{otherwise}.\cr
\end{array}
\right.
\qquad
b = \left\{
\begin{array}{rl}
2 & \mbox{if $k$ is odd and $k<0<l$},\cr
1 & \mbox{if $k$ is odd and $l<0$},\cr
-1 & \mbox{if $k$ is even and $kl>0$},\cr
0 & \mbox{otherwise}.\cr
\end{array}
\right.
$$
If $|l|>2$, then $X(l,l)$ has two components, namely $X_0(l,l)$ of
genus $|n|-1$ and $X_1(l,l)$ of genus $3n^2-7|n|+5$.
\end{theorem}
\begin{proof}
  By Theorem \ref{thmirredcomps} the curve $Y(k,l)$ is geometrically
  irreducible for $k \neq l$, and the curves $Y_0(l,l)$ and $Y_1(l,l)$
  are irreducible if $|l|>2$. Smooth projective completions of these
  curves are $D(k,l)$, $D_0(l,l)$, and $D_1(l,l)$ respectively.  Their
  genera are given in Theorem \ref{genY}. The double cover $X(k,l)$ of
  $Y(k,l)$ is given by $y=x^2-2$.  For $k$ odd, so $k=2m+1$, with
  $|k|>2$, this is equivalent to $t-2 = g_{m+1}(r)^2(x^2-2-r)$ by
  Lemma \ref{tracewk}, or $(g_{m+1}(r)x)^2 = h$ with $h$ as in Lemma
  \ref{hodd}; the fact that $X(k,l)$ is irreducible and the value of
  its genus now follow from Lemmas \ref{ramgenus} and \ref{hodd}.  Now
  assume $k$ is a nonzero even number, so $k=2m$.  Then the double
  cover $X(k,l)$ of $Y(k,l)$ is given by $t-2 =
  (2-r)f_m(r)^2(x^2-2-r)$ by Lemma \ref{tracewk}, or $\big((r-2)
  f_k(r) x \big)^2 = h$ with $h$ as in Lemma \ref{heven}.  If $k\neq
  l$, then the fact that $X(k,l)$ is irreducible and the value of its
  genus follow immediately from Lemmas \ref{ramgenus} and
  \ref{heven}. If $k=l$ and $|l|>2$, then we apply Lemmas
  \ref{ramgenus} and \ref{heven} to both irreducible components of
  $Y(l,l)$ to obtain the final statement.
\end{proof}

\begin{proof}[Proof of Theorem \ref{maintwo}]
This follows immediately from Theorem \ref{genXdiff}.
\end{proof}

Note that from Theorem \ref{genXdiff} we can find all hyperbolic knots
in the family $J(k,l)$ for which the genus of $X_0(k,l)$ equals
$1$. Up to switching $k$ and $l$ and changing sign of both $k$ and
$l$, these are $J(4,4)$ and $J(2,3)$ and $J(-2,2)$, the former of
which is the $7_4$ knot (see \cite[page 391]{rolfsen}), and the latter
two of which are the figure-eight knot.

\section{Commensurability classes}\label{commens}

Recall that a compact orientable
3-manifold $M$ is {\em fibered} if it is homeomorphic to a surface
bundle over $\Sp^1$. One of the most intriguing open conjectures today
is Thurston's virtual fibration conjecture. 

\begin{conjecture}[Thurston]
Every finite-volume hyperbolic 3-manifold has a
finite cover that is fibered. 
\end{conjecture}

Two manifolds are {\em commensurable} if they share a common finite
cover. Since any finite cover of a fibered manifold is fibered,
if one manifold is commensurable to a fibered manifold, then their
common cover is also fibered. It follows that
Thurston's conjecture is equivalent to stating that every
finite-volume hyperbolic 3-manifold is commensurable to a fibered
manifold. 

As knot complements rarely cover each other, it is too much to hope
that every knot complement has a finite cover that is a fibered knot
complement. However, it is natural to ask whether any knot complement
in $\Sp^3$ is commensurable to a fibered knot complement in $\Sp^3$. 
Reid and Walsh \cite{RW} have answered this question negatively for
nonfibered hyperbolic two-bridge knot complements by showing that these
are the unique knot complements in $\Sp^3$ in their commensurability  
class.  

We address the more general question that asks whether a $3$-manifold 
is commensurable to a fibered knot complement in any $\Z/2\Z$-homology
sphere. Calegari and Dunfield \cite{CD} found sufficient conditions 
\cite[Thm. 6.1]{CD} under
which certain hyperbolic knot complements are not commensurable to a
fibered knot complement in a $\Z/2\Z$-homology sphere. 
One consequence \cite[Thm. 7.1]{CD} of their work is that the 
nonfibered two-bridge knots $K(p,q)$ with $0 < p <40$ have complements 
that are not commensurable with a fibered knot
complement in a $\Z/2\Z$-homology sphere.  Hoste and Shanahan
\cite{HS} extended these results to the nonfibered twist knots (the
knots $J(2,n)$ with $n \neq 0, \pm 1, \pm 2$)  and the knots
$J(3,2n)$, for $-33 < n <0$. Our explicit defining equations allow us
to use Calegari and Dunfield's results to prove the following,
conjectured by Hoste and Shanahan \cite[Conj. 1]{HS}.  

\begin{theorem}\label{fibrationthm} 
Let $k,l$ be integers for which the knot $J(k,l)$ has a
nonfibered complement $M$ in $\Sp^3$. Then $J(k,l)$ is hyperbolic and
$M$ is not commensurable to a fibered knot complement in a 
$\Z/2\Z$-homology sphere.
\end{theorem}

Before beginning the proof, we use the Alexander polynomial to
identify the fibered $J(k,l)$ knots. Since these knots are two-bridge
knots, which are alternating, they are fibered if and only
if their Alexander polynomial is monic, i.e., with leading coefficient
$\pm 1$ \cite[Prop. 13.26]{BuZ}. 
Without loss of generality we may assume that
$l$ is even.  We can easily compute the Alexander polynomials. 

\begin{lemma}\label{alexpoly} For  all nonzero integers $k$ and
  $l=2n$, the Alexander polynomial $\Delta_{k,l}(t)$ of the knot $J(k,l)$ is
\begin{enumerate}
\item $ \displaystyle  nmt^2+(1-2nm)t+nm $ \\ if $k=2m$,
\item $ \displaystyle mt^{2n}+(1+2m)( -t^{2n-1}+\dots -t)+m$
\\ if $k=2m+1$ and $l>0$, and 
\item $ \displaystyle (m+1)t^{-2n} +(1+2m)(-t^{-2n-1}+ \dots -t)+(m+1) $ 
\\ if $k=2m+1$ and $l<0$.
\end{enumerate}
\end{lemma}
It follows that $J(k,2n)$ is fibered only for the unknot $J(0,l)=J(k,0)$,
the figure-eight $J(2,-2)=J(-2,2)$, the trefoil $J(2,2)=J(-2,-2)$,
the knots $J(3,2n)=J(-3,-2n)$ for any  $n>0$ and $J(1,2n)=J(-1,-2n)$ for any
$n$. 

\begin{proof}[Proof of Theorem \ref{fibrationthm}]
First note that $J(k,l)$ is hyperbolic, as the only
nonhyperbolic knots of the form $J(k',l')$ are the torus
knots $J(\pm 1, 2n)$, the unknot and the trefoil
$J(2,2) = J(-2,-2)$ (see \cite[Thm. 1]{HT}).  

Also note that $M= \Sp^3 \setminus J(k,l)$ is not arithmetic, as the 
only arithmetic knot is the fibered figure-eight knot;
for the definition of 
arithmetic and the proof of this fact, see \cite{Reid} and 
\cite[Section 9.4]{MR}. We will
show that $M$ is in fact {\em generic}, which for a $1$-cusped
hyperbolic $3$-manifold 
means that it is not arithmetic
and its commensurator orbifold has a flexible cusp, i.e., a cusp that
is not rigid. See \cite[section 2.1]{RW} for an explanation of the 
latter condition, which for any hyperbolic complement $M'$ of a
nonarithmetic knot is equivalent with the fact that
$M'$ has no hidden symmetries (isometries of a finite cover of $M'$
that are not the lift of an isometry of $M'$) by \cite[Prop. 9.1]{NR}.  
Reid and Walsh show that the complement of no hyperbolic two-bridge knot
other than the figure-eight has hidden symmetries \cite[Thm. 3.1]{RW}. 
We conclude that $M$ is indeed generic. 

A representation $\rho:\pi_1 (M) \to \PSL_2(\C)$ is called {\em
integral} if for all $\gamma \in \pi_1 (M)$ the trace $\tr
(\overline{\rho(\gamma)})$ of a lift $\overline{\rho(\gamma)} \in
\SL_2(\C)$ of $\rho(\gamma)$ is an algebraic integer.
Calegari and Dunfield \cite[Thm. 6.1]{CD} prove that if $M'$ is a generic
hyperbolic knot complement in a $\Z/2\Z$-homology sphere and if
$Y_0\big(\pi_1(M')\big)$ contains the character of a nonintegral reducible
representation, then $M'$ is not commensurable to a fibered knot
complement in a $\Z /2\Z$-homology sphere. 
From the above discussion, it suffices to show that the component of 
$C(k,l)$ (see \S \ref{equationsection}) corresponding to $Y_0(k,l)$ 
contains the character of a nonintegral reducible representation. 
Without loss of generality we will assume that $l=2n$ is even. We 
use the notation from \S \ref{equationsection}. A representation is
reducible exactly when $r=2$.  The points $(r,y) \in C(k,l)$ with $r=2$ 
satisfy $F(y)=0$ with
\[ F(y)=f_n(t)\Big( \Phi_{-k}(2)\Phi_{k-1}(2)(y-2)-1 \Big) +f_{n-1}(t)\]
and 
\begin{equation}\label{tiny}
t=\Phi_{-k}(2) \Psi_k(2)(y-2)+2. 
\end{equation}
By Theorem \ref{thmirredcomps} the curve $C(k,l)$ is irreducible
unless $k=l$. By Lemma \ref{fgat2} we have $\Phi_{2j}(2)=f_j(2)=j$, 
$\Phi_{2j+1}(2)=1$, $\Psi_{2j}(2)=0$, and $\Psi_{2j+1}(2)=1$ for
all integers $j$. 

First, consider the case where $k$ is odd, say $k=2m+1$. Then $k\neq
l$, so $C(k,l)$ is irreducible. Here, $t=y$.  Then 
\[ \begin{aligned} F(t) & =
 f_n(t)(m(t-2)-1)+f_{n-1}(t)\\ &
  =mf_{n+1}(t)-kf_n(t)+(m+1)f_{n-1}(t), 
\end{aligned}\] 
where we used $tf_n(t) = f_{n-1}(t)+f_{n+1}(t)$ in the last inequality.
For any integer $j$, the constant terms of $f_{2j}$ and $f_{2j+1}$ are 
$0$ and $(-1)^{j}$ respectively. Therefore, the constant term of 
$F(t)$ is $\pm 1$ if $n$ is even and $\pm k$ if $n$ is odd. 
The leading term is $m$ if $l$ is positive and $m+1$ if $l$ is negative.  
As $k=2m+1$, we conclude that in all cases the leading term and
constant term are relatively prime. Therefore, $F(t)$ has a
nonintegral root exactly when the leading term is not $\pm 1$.  
The leading term is $1$ only when $m=1$
$(k=3)$ and $l>0$ or when $m=0$ ($k=1$) and $l<0$.  It is $-1$ only when
$m=-1$ ($k=-1$) and $l>0$ or when $m=-2$ ($k=-3$) and $l<0$. All cases 
correspond to fibered knots by Lemma \ref{alexpoly}, so we conclude
that $F(t)$ does have a nonintegral root $y_0$ corresponding to a
nonintegral point $(2,y_0)$ on $C(k,l)$ and thus on $Y_0(k,l)$. 

Now it suffices to assume $k$ is even, say $k=2m$. From (\ref{tiny})
we get $t=2$. By Lemma \ref{specpointsoldboth} there is a
unique point $P=(2,2-1/mn)$ on $C(k,l)$ with $r=2$.  If $k\neq l$ then
$C(k,l)$ is irreducible, so $P$ corresponds to a nonintegral point on
$Y_0(k,l)$.  If $k=l$, then the birational morphism to the new model
$D(k,k)$ (see Proposition \ref{charvar})  sends $P$ to $(2,2)$, which
lies on the component corresponding to $Y_0(k,k)$ by Proposition
\ref{Yzero}.  We conclude that $P$ is a nonintegral point on
$Y_0(k,k)$ in this case as well. 
\end{proof}

\begin{proof}[Proof of Theorem \ref{mainfour}]
Given that every manifold is commensurable with itself,
this follows immediately from Theorem \ref{fibrationthm}.
\end{proof}

\small
\nocite{*}
\bibliography{charvarpap}     
\bibliographystyle{plain}  

\end{document}